\documentclass[11pt]{article}

\usepackage[english]{babel}
\usepackage[letterpaper,top=2cm,bottom=2cm,left=2.5cm,right=2.5cm,marginparwidth=1.75cm]{geometry}
\usepackage{amsmath}
\usepackage{amsthm}
\usepackage{thmtools}
\usepackage{amssymb}
\usepackage{soul}
\usepackage{mathtools}
\usepackage{graphicx}
\usepackage{xcolor}
\usepackage{subcaption}
\usepackage{comment}
\usepackage[normalem]{ulem}
\usepackage[hypertexnames=false]{hyperref}
\hypersetup{
  colorlinks   = true, 
  urlcolor     = magenta, 
  linkcolor    = blue, 
  citecolor   = [RGB]{20 139 34}
}

\usepackage{zref-clever}
\let\oldnewtheorem\newtheorem

\RenewDocumentCommand{\newtheorem}{momo}{
  \IfValueTF{#2}{
    \AddToHook{env/#1/begin}{
      \zcsetup{countertype={#2=#1}}}
      \zcRefTypeSetup{#1}{
Name-sg = #3 ,
      }
    \oldnewtheorem{#1}[#2]{#3}
  }{
    \AddToHook{env/#1/begin}{
      \zcsetup{countertype={#1=#1}}}
    \zcRefTypeSetup{#1}{
Name-sg = #3 ,
      }
    \IfValueTF{#4}{
      \oldnewtheorem{#1}{#3}[#4]
    }{
      \oldnewtheorem{#1}{#3}
    }
  }
}

\newcommand{\cref}[1]{\zcref{#1}}
\newcommand{\Cref}[1]{\zcref[S]{#1}} 

\usepackage[all]{xy}
\usepackage{dirtytalk}

\title{Null distance on cosmological spacetimes and monotone convergence}
\author{Miguel Prados-Abad\texorpdfstring{\thanks{
    \parbox[t]{.9\linewidth}{
          Faculty of Mathematics, University of Vienna, Oskar-Morgenstern-Platz 1, 1090 Vienna\\
          {\em Emails}: \href{mailto:miguel.prados.abad@univie.ac.at}{\texttt{miguel.prados.abad@univie.ac.at}},\,
                 \href{mailto:omar.zoghlami@univie.ac.at}{\texttt{omar.zoghlami@univie.ac.at}}%
        }}
    }{}
     \and Omar Zoghlami\footnotemark[\value{footnote}]
}
\date{}

\newcommand{\R}{\mathbb{R}}
\newcommand{\N}{\mathbb{N}}
\newcommand{\abs}[1]{ \left\vert #1 \right\vert}
\renewcommand{\d}{{\smash{\hat{d}}}}
\renewcommand{\L}{\smash{\hat{L}}}

\usepackage{latexsym}
\usepackage[utf8]{inputenc}
\usepackage{amsfonts,fontaxes}
\usepackage{xspace,float}
\usepackage{color,colortbl,tikz}
\usepackage{xfrac}
\usepackage{bm}
\usepackage{cancel}
\usepackage[backend=biber, style=alphabetic, maxcitenames=1,uniquelist=false,maxbibnames=20]{biblatex}
\usepackage{csquotes,titlesec,enumitem}
\usepackage{multicol}

\AtEveryBibitem{\clearfield{issn}\clearfield{isbn}\clearfield{url}}

\theoremstyle{plain}
\newtheorem{theorem}{Theorem}[section]
\newtheorem{lemma}[theorem]{Lemma}
\newtheorem{proposition}[theorem]{Proposition}
\theoremstyle{remark}
\newtheorem{remark}[theorem]{Remark}

\theoremstyle{plain}
\newtheorem{corollary}[theorem]{Corollary}

\theoremstyle{definition}
\newtheorem{definition}[theorem]{Definition}

\theoremstyle{remark}
\newtheorem{example}[theorem]{Example}

\numberwithin{equation}{section}

\newcommand{\diam}{\mathrm{diam}}

\begin{document}
\maketitle
\begin{abstract}
The metric theory of spacetimes studies Lorentzian manifolds using tools of metric geometry. This is achieved via the null distance, which is a definite distance constructed from a time function on a spacetime. This enables the study of Gromov-Hausdorff-type convergence of spacetimes, a program recently initiated by Sakovich and Sormani. In this paper we study such notions of convergence for cosmological spacetimes with compact slices, i.e.,  $(a,b)\times M$ endowed with a Lorentzian metric $-dt^2+h_t$, where $h_t$ is a family of Riemannian metrics on the compact manifold $M$. Assuming mild extension properties of $h_t$, we first establish that these spacetimes are causally-null compactifiable and future developed. We then study monotone sequences with a uniform upper bound on the spatial diameter, obtaining uniform convergence of the null distances, as well as convergence of the associated timed metric spaces in the future developed Gromov-Hausdorff sense. Finally, we prove that causally-null compactifiable spacetimes satisfying a mild causal accessibility condition are causally-null, and relate the causally-null distance induced by the limit distance with the null distance induced by the (possibly non-smooth) limit metric tensor. Examples are provided to motivate the necessity of our hypotheses.

\end{abstract}



\tableofcontents

\section{Introduction}

The metric theory of spacetimes seeks to study Lorentzian geometry using the tools of metric geometry: convergence, compactness, and synthetic notions of curvature and dimension that have proven powerful in the Riemannian setting. In the Lorentzian framework, the central obstruction is the absence of a canonical distance: the Lorentzian distance, also called time-separation function is degenerate, fails to be symmetric and satisfies the \emph{reverse} triangle inequality.

The null distance framework, introduced in the work by Sormani and Vega \cite{sormani_null_2016}, has emerged as a versatile tool to study the geometry of spacetimes with metric techniques. It replaces reliance on Lorentzian length with a \textit{null length}, built from a (generalized) time function $\tau$ by measuring the length of causal curves through the variation of $\tau$ alone. This length functional provides a genuine pseudo-distance obeying the
ordinary triangle inequality. Moreover, as shown in the work by Sormani and Sakovich \cite{sakovich_null_2023}, under mild hypotheses on the time function such as local anti-Lipschitzness, the null distance is definite and encodes causality, thereby turning a spacetime into a metric space while keeping its causal information. Furthermore, in \cite[Theorem 1.3]{sakovich_null_2023} the authors proved that the null distance induced by the cosmological time of Andersson, Galloway and Howard \cite{andersson_cosmological_1998} retains the full causal information about a spacetime: bijections preserving both the null distance and the cosmological time are in fact Lorentzian isometries.

In a recent key contribution, Sakovich and Sormani \cite{sakovich_introducing_2024} have turned the study of convergence of metric spacetimes into a research program. Working within the class of so-called causally-null compactifiable spacetimes (\Cref{def:null_compactifiable}), they canonically convert them into compact timed metric spaces via the cosmological time and the induced null distance. They also defined the notions of  \emph{timed-Hausdorff} and \emph{future developed Gromov-Hausdorff} distance and proved definiteness results within this class, characterizing the zero-distance case as time-preserving Lorentzian isometry. In the same spirit, they defined notions of convergence to limit spacetimes that need not be smooth.

More recently, Perales showed that future developed convergence implies timed-Hausdorff convergence \cite{perales_timed_hausdoff_2025}, and Che and Perales \cite{che_perales_2026} introduced \emph{causally-null} timed metric spaces, asking under which conditions the distance in a timed metric space is a null distance (\Cref{def:causally_null_distance}). Alternative approaches to convergence of spacetimes and their low regularity counterparts have recently been proposed: for instance in the works of Minguzzi and Suhr \cite{minguzzi_lorentzian_2024}; Bykov, Minguzzi and Suhr \cite{bykov_lorentzian_2025}; and Mondino and S\"amann \cite{mondino_lorentzian_2025}.

By now a body of convergence results has grown in the framework of \cite{sormani_null_2016}: Allen and Burtscher \cite{allen_properties_2022} gave warped products an integral current structure and proved Gromov-Hausdorff and intrinsic-flat convergence under uniform convergence of the warping functions; Allen \cite{allen_null_2023} relaxed this to $L^1$-convergence with a lower bound; Kunzinger and Steinbauer \cite{kunzinger_null_2022} extended these concepts to the framework of Lorentzian length spaces; and Burtscher and García-Heveling \cite{burtscher_global_2024} related the null distance to global hyperbolicity. 

In this contribution we carry further the research program started in \cite{sakovich_introducing_2024} by considering \emph{cosmological spacetimes} (\Cref{def:generalized_product}), where the spacetime splits as \((0, \tau_{\max}) \times M\) with Lorentzian metric \(g = -dt^2 + h_{(t,x)}\), where \(h\) has Riemannian signature. This class is closely related to some of the most relevant examples of spacetimes. For instance, globally hyperbolic spacetimes are always conformal to a cosmological spacetime (see the works of Bernal and Sánchez \cite{bernal_smooth_2003, bernal_smoothness_2005}). Warped product spacetimes fall into this category as well, and future developments of compact initial data sets \((M,h,k)\) are always locally cosmological spacetimes near the initial data set as shown by Choquet-Bruhat and Geroch \cite{Choquet-Bruhat1969}. More precisely, we will be considering cosmological spacetimes where \(M\) is a compact manifold and satisfying the following assumption:
\begin{equation}\label{eq:A}
     \text{The metric tensor } h \text{ extends continuously to } t=0 \text{ and } t=\tau_{\max} \text{ as a Riemannian metric}. \tag{A}
\end{equation}
For these spacetimes the cosmological time is simply $\tau=t$ which is automatically regular when the slices are compact (cf.~\Cref{lemma:regular_cosmotime_when_compact_slices}). Thus the null distance \(\d_g\) with respect to \(\tau\) is definite and encodes causality by \cite{sakovich_null_2023}. Hypothesis \eqref{eq:A} encodes the reading of these spaces as future developments of initial data.

We first establish in \Cref{thm:compactness_nullcompletion} that, under natural assumptions on their Lorentzian metrics, cosmological spacetimes with compact slices (cf.~\Cref{def:generalized_product}) are causally-null compactifiable in the sense of \cite[Def. 1.1]{sakovich_introducing_2024}. More precisely, we prove that the metric completion of the associated metric spaces endowed with the null distance \(\d_g\) is actually bi-Lipschitz equivalent to the topological product \([0, \tau_{\max}] \times M\) endowed with some taxi distance:
\begin{theorem}\label{thm:compactness_nullcompletion}
    Let $N=(0,\tau_{\max})\times M$ be a cosmological spacetime with compact slices satisfying assumption \eqref{eq:A}. Then, the metric completion \((\bar{N},\d_g)\) of the space \((N,\d_g)\) is compact. Moreover, fixing \(t_0 \in [0, \tau_{\max}]\), \((\bar{N},\d_g)\) is bi-Lipschitz equivalent to the space \(([0,\tau_{\max}] \times M, d)\), where
        \[
        d((t,x),(s,y)) = \abs{t-s} + d_{t_0}(x,y),
        \]
    where \(d_{t_0}\) is the distance induced on \(M\) by the Riemannian metric $h_{t_0}=h_{\vert \{t_0\}\times M}$.
\end{theorem}
\noindent
This result sharpens \cite[Theorems 3.1, 3.10]{sakovich_introducing_2024}, identifying explicitly the topology of the completion without the requirement that the metric extends smoothly \say{below} the initial data set. An important ingredient used in the proof is a result proven by Nigri in \cite[Theorem 1.2]{nigri2025nulldistancetemporalfunctions}, for which we present an alternative proof.

We then explore timed-Hausdorff and future developed convergence (cf.\ \cite[Definition~4.22]{sakovich_introducing_2024} and \Cref{def:fd_gh_convergence}) of sequences of cosmological spacetimes, with fixed base manifold and varying Lorentzian metrics. In the spirit of \cite{perales_monotone_2025}, we give a sufficient condition for the family of induced null distances to be monotonically increasing and converging uniformly to a limit distance: we impose that the \say{Riemannian} components of the Lorentzian metric, with respect to the splitting induced by the cosmological spacetime structure, are monotonically increasing and uniformly bounded. More precisely, we assume that:
\begin{equation}\label{eq:B}
\begin{rcases}
    1.&\text{Each element of the sequence satisfies \eqref{eq:A}};\\
    2.&\text{The sequence is monotone, namely } \forall j \in \N, \forall t \in [0,\tau_{\max}],\: h_{j,t} \leq h_{j+1,t}\:;\\
    3.&\text{There exists a Riemannian metric } \tilde{h} \text{ on } M \text{ such that for each } j \in \N \text{ and }\quad\\ 
    & t\in[0,\tau_{\max}]\text{ we have } h_{j,t} \leq \tilde{h} \text{ on } TM.
\end{rcases}\tag{B}
\end{equation}\label{itemize:list_of_assumptions}
Under this set of assumptions we get the following result:

\begin{theorem}\label{thm:future_developed_convergence}
    Let $(N,g_j)$ be a sequence of cosmological spacetimes, i.e., 
    \[
    N=(0,\tau_{\max})\times M, \qquad g_j=-dt^2+h_{j,(t,x)}
    \]
    satisfying $h_j(\partial_t,v)=0$ for every $v\in TN$ and satisfying assumption \eqref{eq:B}.
    Denote by $d_\infty$ the distance obtained as pointwise limit of the null distances $\d_j$. Then the convergence of the distances is actually uniform and the sequence of associated timed metric spaces converges in the future developed Gromov-Hausdorff and timed-Hausdorff sense to the compact space $([0,\tau_{\max}] \times M, d_\infty)$.
 \end{theorem}

 We remark that there exists a pre-compactness result for timed-Hausdorff convergence presented in \cite{che-perales-sormani} which holds for a more general class of timed metric spaces. In our setting, \Cref{thm:future_developed_convergence} actually gives convergence of the full sequence, along with information about the topology of the limit space.

We then shift our attention to what the \emph{causally-null} theory of \cite{che_perales_2026} sees in the limit. Firstly, we give a sufficient \say{causal connectivity} condition for a general timed metric space to have a causally-null completion (\Cref{prop:causally_null_completions}) and, as a consequence, we obtain the following:
\begin{theorem}\label{thm:cosmological_spacetimes_are_causally_null}
    Let \((N,g)\) be a cosmological spacetime with compact slices satisfying assumption \eqref{eq:A}. Then its associated metric space \((\bar{N},\d)\) is causally-null.
\end{theorem}

Subsequently, we characterize the causally-null distance induced by the limit space of sequences as above as the extension to the metric completion of the null distance associated to the (non necessarily smooth) limit metric: 
\begin{theorem}\label{thm:causally_null_limit_is_null_distance_of_limit_tensor}
    Let \(\{(N,g_i)\}_{i \in \N}\) be a sequence of cosmological spacetimes with compact slices satisfying assumption \eqref{eq:B}. Consider both the limit distance \(d_\infty\) given by \Cref{thm:future_developed_convergence} and the null distance \(\d_\infty\)  associated to the limit metric \(g_\infty\). Then the causally-null distance \(\hat{d}_{d_\infty,\tau}\) coincides with the extension of the null distance \(\d_\infty\) to the completion \((\bar{N},\d_{\infty})\).
\end{theorem}
\noindent
This is subtle because, as \Cref{ex:AB-bis} exhibits, the limit distance can fail to be causally-null. The previous result shows that in this setting the notion of causally-null distance is able to retain information about the limit tensor, which the null distance forgets in the limit.

In the last part of this work we discuss the necessity of the assumptions we set for our results with some examples. Firstly in \Cref{ex:allen-burtscher} we discuss how one can apply \Cref{thm:future_developed_convergence} to a sequence of warped products whose warping functions are not converging uniformly, improving on \cite[Theorem 1.4]{allen_properties_2022}. Then in \Cref{ex:limit_non_compact} we show that without the assumption \eqref{eq:B} in \Cref{thm:uniform_convergence} limit spaces of such sequences need not be homeomorphic to a product. Finally, \Cref{ex:triangle} and \Cref{ex:Cantor} motivate the requirement of the accessibility condition of \Cref{prop:causally_null_completions} and the absolutely continuous regularity class of curves in \Cref{prop:causal-relation-induced-by-limit-tensor}, respectively.

\subsection*{Acknowledgements}
 The authors would like to thank Prof.\ Christina Sormani for introducing them to the topic and for the early discussions and meetings, Prof.\ Raquel Perales for the time she dedicated to meetings, revisions, and insightful suggestions that improved this paper, Prof.\ Roland Steinbauer for carefully reading the manuscript, and all of them for many helpful discussions. The authors gratefully acknowledge support from the Simons Center for Geometry and Physics, Stony Brook University, at which part of the research for this paper was performed. This research was funded by the Austrian Science Fund (FWF) [Grant DOI \href{https://www.fwf.ac.at/forschungsradar/10.55776/EFP6}{10.55776/EFP6}]. For open access purposes, the authors have applied a CC BY public copyright license to any author accepted manuscript version arising from this submission.
 
\section{Preliminaries}

\subsection{Spacetimes}

Let $N$ be a smooth manifold endowed with a Lorentzian metric $g$ of signature $(-,+,\ldots,+)$. We call tangent vectors $v\in TN$ \textit{timelike} (resp.\ \textit{spacelike}) whenever $g(v,v)<0$ (resp.\ $g(v,v)>0$). We say that $v\in TN$ is \textit{null} if $v\neq 0$ and $g(v,v)=0$. A vector is said to be \textit{causal} if it is timelike or null. Similarly, a piecewise smooth curve is said to be \textit{causal (null, timelike, spacelike)} if its tangent vector at every point of differentiability is so.

Causal vectors at a point form a set with two connected components called \textit{cones}. A globally defined continuous timelike vector field on $N$ defines a choice of cones called \textit{time orientation}. The chosen cone at each point is called \textit{future cone} and a causal vector is said to be \textit{future directed} if it lies in the future cone and \textit{past directed} otherwise. A connected Lorentzian manifold endowed with a time orientation is called a \textit{spacetime}. 

A causal curve is \textit{future directed} if all its tangent vectors are future directed. Two points $p,q\in N$ are said to be \textit{timelike related}, which is denoted by $p\ll q$, if there exists a future directed timelike curve from $p$ to $q$. Similarly, they are said to be \textit{causally related}, denoted $p\leq q$, if there is a future directed causal curve joining them or $p=q$. As usual, causal (resp.\ timelike) pasts and futures are denoted by $J^\pm(p)$ (resp.\ $I^\pm(p)$). The set of all timelike diamonds $I^+(p)\cap I^-(q)$ forms a subbasis for a topology in $N$, called Alexandrov topology. A spacetime is said to be \textit{causal} if the causal relation is antisymmetric, i.e., if $p\leq q$ and $q\leq p$ then $p=q$; \textit{strongly causal} if the Alexandrov topology coincides with the manifold topology.

A function $\tau\colon N\to \R$ which is strictly increasing along every future directed causal curve is called a \textit{generalized time function}. If $\tau$ is continuous, it is called a \textit{time function}.

A \textit{Cauchy hypersurface} is a subset $S\subset N$ which intersects every inextendible timelike curve exactly once. If a Cauchy hypersurface exists, then there is also a smooth and spacelike one, and therefore every inextendible \textit{causal} curve intersects it only once. A spacetime is said to be \textit{globally hyperbolic} if it is causal and the causal diamonds $J^+(p)\cap J^-(q)$ are compact or, equivalently, if it admits a Cauchy hypersurface.

It is now a classical result \cite{bernal_smooth_2003, bernal_smoothness_2005} that a globally hyperbolic spacetime $N$ is isometric to a manifold $\R\times M$ with Lorentzian metric $\smash{g_{(t,x)}=-f(t,x)dt^2+h_{(t,x)}}$, where $f\colon \R\times M\to (0,\infty)$ is smooth and $h$ is another metric on the product, satisfying:
\begin{enumerate}
    \item\label{item:nabla} $\nabla t$ is past directed timelike,
    \item \label{item:slices_riemannian} Each level set $M_t$ of $t$ is a Cauchy hypersurface and the restriction $h_{|M_t}=:h_t$ of $h$ to any of them is Riemannian, and
    \item\label{item:radical} The radical of $h$ is given by $\mathrm{span}(\nabla t)$, i.e., $h(\nabla t, v)=0$ for every $v\in T(\R\times M)$.
\end{enumerate}
In the special case in which, additionally, the function $f$ has no dependence on $x$, a change of variables allows one to obtain a (Lorentzian) isometry between $N$ and a product manifold
\begin{equation}\label{eq:gen_product}
(a,b)\times M, \quad \text{with metric }\: g_{(t,x)}=-dt^2+h_{(t,x)},
\end{equation}
where $-\infty\leq a<b\leq \infty$. Notice that every globally hyperbolic spacetime is conformally related to a space as in \eqref{eq:gen_product} via the conformal factor \(1/f\).

\begin{definition}\label{def:generalized_product}
    We say that a spacetime \((N,g)\) is a {\em cosmological spacetime} if \(N = I \times M\), where $I=(t_1,t_2)\subset\R$ is bounded, and its Lorentzian metric splits as
    \[
        g = -dt^2 +h_{(t,x)},
    \]
    where \(h\) satisfies \Cref{item:radical} above. We say that $N$ has {\em compact slices} if \(M\) is compact.
\end{definition}
Notice that such spacetimes are not assumed to be globally hyperbolic. \Cref{item:nabla} is trivially satisfied in this case. About \Cref{item:slices_riemannian}, it is not necessarily true that the level sets \(M_t\) are Cauchy hypersurfaces but the validity of \Cref{item:radical} ensures that each \(h_t\) is Riemannian, since the signature of \(g\) is given by \((-, +, \dots, +)\).

In this way, we distinguish these from \textit{warped products} (in which $\smash{h_{(t,x)}=f(t)h_x}$, i.e., the restrictions $h_t$ of $h$ to $M_t$ are all proportional to a fixed metric on $M$ via the projection) and \textit{products} (i.e., warped products in which $f\equiv 1$, i.e., $h$ does not depend on $t$).

\subsection{Null distance}

We now recall some notions introduced in \cite{sormani_null_2016}.

\begin{definition}\label{def:piecewise-causal}
    Let $(N,g)$ be a spacetime. A \textit{piecewise causal} curve in $N$ is a curve $\alpha\colon [a,b]\to N$ such that there is a partition $a=t_0<\ldots<t_n=b$ with $\alpha_{|(t_{i-1},t_{i})}$ smooth causal, for every $i$. Different segments may have different time-orientations.
\end{definition}

\begin{definition}\label{def:null-distance}
    Let $(N,g)$ be a spacetime endowed with a generalized time function $\tau$. 
    \begin{enumerate}
    \item Let $\alpha\colon [a,b]\to N$ be a piecewise causal curve and let $a=t_0<\ldots<t_n=b$ be a partition as in \Cref{def:piecewise-causal}. Call $\alpha_i=\alpha(t_i)$. The \textit{null length} of $\alpha$ is
    \[
    \hat{L}_\tau(\alpha)=\sum_{i=1}^n \bigl|\tau(\alpha_i)-\tau(\alpha_{i-1})\bigr|.
    \]
    \item Given two points $p,q\in N$, the \textit{null distance} between them is
    \[
    \d(p,q)=\d_{g,\tau}(p,q)=\inf
    \bigl\{
    \hat{L}_\tau(\alpha) \mid \alpha \text{ is piecewise causal from } p \text{ to } q
    \bigr\}.
    \]
    \end{enumerate}
\end{definition}

Despite its name, the null distance is, in general, not a distance. Nevertheless, it is a \textit{pseudo-distance} \cite[Lemma~3.8]{sormani_null_2016}, i.e., it satisfies all properties of a distance except, at most, \textit{definiteness}. In other words, it is possible that $\d(p,q)=0$ whereas $p\neq q$. However, in the cited article \cite{sormani_null_2016}, the authors discuss a property of the generalized time function that ensures that the null distance it induces is definite and, therefore, a distance.
\begin{definition}
    Let $(N,g)$ be a spacetime and $f\colon N\to \R$. We say that $f$ is \textit{locally anti-Lipschitz} if for every point $p\in N$ there is a neighborhood $U\ni p$ and a distance $d_U$ such that the following holds on $U$:
    \[
    q\leq r\implies d_U(q,r)\leq f(r)-f(q).
    \]
\end{definition}

Coming back to the properties of the null distance, it is evident that $\d(p,q)\geq \abs{\tau(q)-\tau(p)}$ and that equality holds whenever $p$ and $q$ are causally related. The converse implication is not always true.

\begin{definition}[{\cite{sormani_null_2016,sakovich_null_2023}}]
\label{def:causality_encoding}
    The null distance $\d$ is said to \textit{encode causality} if for every pair of points $p,q$ one has
    \[
    \d(p,q)=\tau(q)-\tau(p)\iff p\leq q.
    \]
\end{definition}

Some authors have given conditions that ensure that the null distance encodes causality. This is the case, for instance, if the time function is locally anti-Lipschitz and proper \cite[Theorem~4.1]{sakovich_null_2023}. See also the works of Burtscher and García-Heveling \cite[Theorem~1.9]{burtscher_global_2024}, and of Galloway \cite[Theorem~3]{galloway_note_2023} for refinements of that result with weaker assumptions.

In the foregoing we will consider a specific time function called \textit{cosmological time}, introduced by Andersson, Galloway and Howard in \cite{andersson_cosmological_1998}, which under some regularity assumptions induces a definite null distance encoding causality.

\begin{definition}
    Let $(N,g)$ be a smooth spacetime and denote by $\ell$ its time separation function. The \textit{cosmological time} is the map $\tau_g\colon N\to [0,\infty)$ defined for every point $q\in N$ by
    \[
    \tau_g(q)=\sup\{\ell(p,q)\mid p\leq q\}.
    \]
    We say that $\tau_g$ is \textit{regular} if it is finite and $\tau_g\to 0$ along every past directed inextendible causal curve.
\end{definition}
\begin{proposition}[{\cite[Theorem~1.2, Corollary~2.6]{andersson_cosmological_1998}}]\label{prop:properties-regular-cosmological-time}
    Let $(N,g)$ be a smooth spacetime with regular cosmological time $\tau_g$. Then $(N,g)$ is globally hyperbolic, $\tau_g$ is a locally Lipschitz\footnote{i.e., given a compact neighborhood of a point, the restriction of $\tau_g$ to the neighborhood is Lipschitz with respect to the distance induced by \textit{any} auxiliary Riemannian metric defined on it, as all such distances are bi-Lipschitz equivalent.} time function and, moreover, its gradient exists almost everywhere and is (past directed) timelike and locally bounded away from the light cones. Moreover, the level sets $M_t=\tau_g^{-1}(t)$, if non-empty, are Cauchy hypersurfaces.
\end{proposition}

\begin{proposition}[{\cite[Theorem~4.18]{sormani_null_2016}}]\label{prop:local-antiLipschitness}
    If a time function has timelike gradient defined almost everywhere and locally bounded away from the light cones, then it is locally anti-Lipschitz, and thus induces a definite null distance. In particular, every regular cosmological time is locally anti-Lipschitz.
\end{proposition}

Coming back to the \textit{causality encoding} property, in most cases we will be in the following situation:

\begin{proposition}[{\cite[Theorem~1.9]{burtscher_global_2024}}] \label{prop:encoding_causality}
    If all nonempty level sets of a locally anti-Lipschitz time function are future (or past) Cauchy (a notion strictly weaker than a Cauchy surface), then the null distance encodes causality.
\end{proposition}

Notice that, in particular, \Cref{prop:properties-regular-cosmological-time} ensures that this is the case for a regular cosmological time. For cosmological spacetimes (see \Cref{def:generalized_product}), the cosmological time is given by a translation of $t$: $\tau_g=t-t_1$. For convenience, we will usually take $t_1=0$ after a translation of $t$, and therefore $\tau_g=t$. We will thus denote $t_2=\tau_{\max}$. When the slices are compact, $\tau_g$ is proper and \Cref{lemma:regular_cosmotime_when_compact_slices} gives regularity, so that we can even apply \cite[Theorem~4.1]{sakovich_null_2023} to obtain causality encoding.

In the foregoing, for a given spacetime $(N,g)$ with regular cosmological time $\tau_g$, the null distance will be assumed to be taken with respect to the cosmological time. We will therefore denote it $\d_g$ or, simply, $\d$. Notice that, following previous discussions, such a null distance will automatically be definite and encode causality.

\begin{definition}[{\cite[Def.~1.1]{sakovich_introducing_2024}}]\label{def:null_compactifiable}
    A spacetime $(N,g)$ is \textit{causally-null compactifiable} if (1) it has a bounded regular cosmological time function $\tau_g$,
    and (2) the metric completion $(\bar{N},\d_g)$ of $(N,\d_g)$ is compact. We call \textit{associated timed-metric space} the triple $(\bar{N},\d_g,\tau_g)$, where $\d_g$ and $\tau_g$ are the (unique) extensions of $\d_g$ and $\tau_g$, respectively, to the metric completion; we use the same notation for simplicity.
\end{definition}

\begin{definition}[{\cite[Def.~3.9]{sakovich_introducing_2024}}]\label{def:future-developed}
    A causally-null compactifiable spacetime $(N,g)$ is \textit{future developed} if its associated timed metric space $(\bar{N},\d_g, \tau_g)$ has extended cosmological time function $\tau_g\colon \bar{N}\to [0,\tau_{\max}]$ with \textit{initial} level set $M\coloneqq \tau_g^{-1} (0)$ satisfying the \textit{distance from initial data property}, i.e., 
    \[
    \tau_g(p)=\min\bigl\{
    \d_g(p,q)\mid q\in M
    \bigr\},\qquad \forall p\in \bar{N}.
    \]
\end{definition}

Spacetimes with regular cosmological time fall into the more general category of {\em timed metric spaces}: 
\begin{definition}[{\cite[Def. 2.5]{che-perales-sormani}}]
    A {\em timed metric space} is a triple \((X,d,\tau)\) where \((X,d)\) is a metric space and \(\tau \colon X \to [0,+\infty)\) is a \(1\)-Lipschitz function, called {\em time function}.
\end{definition}

In this setting we can also extend the notion of future developments to {\em future developed} timed metric spaces:
\begin{definition}[{\cite[Def. 2.6]{perales_timed_hausdoff_2025}}]\label{def:timed_metric_space_future_developed}
    A timed metric space is {\em future developed} if by setting \(A \coloneqq \tau^{-1}(0)\) we have that
    \[
        \abs{A}\geq 2 \qquad \text{and} \qquad \tau(p) = \operatorname{dist}(p,A).
    \]
\end{definition}

Of course, the associated timed metric space of a future developed spacetime in the sense of \Cref{def:future-developed} is future developed in the sense of \Cref{def:timed_metric_space_future_developed}.

\subsection{Monotone convergence and convergence of metric pairs}
In this section we discuss some notions of convergence of metric spaces which will be used in the sequel. 

First, we introduce some results from \cite{perales_monotone_2025}. The first of them deals with embedding two metric spaces with the same underlying set into a common metric space. 
\begin{proposition}[{\cite[Proposition~3.7]{perales_monotone_2025}}]\label{prop:embedding_taxi}
    Let $(X,d_a)$ and $(X,d_b)$ be two metric spaces with the same underlying set and let $\varepsilon>0$ be such that
    \[
    d_b(x,y)-\varepsilon\leq d_a(x,y)\leq d_b(x,y),\qquad \forall x,y\in X.
    \]
    Then there is a metric space $(Z=X\times[0,\varepsilon/2],d_Z)$, where the distance between points $z_i=(x_i,t_i)$ is given by 
    \[
    d_Z(z_1,z_2)\coloneqq\min\bigl\{d_b(x_1,x_2)+\abs{t_1-t_2},\: \varepsilon-t_1-t_2+d_a(x_1,x_2)\bigr\},
    \]
    and the maps 
    \[
        \begin{aligned}
            f_a&\colon (X,d_a)\to (Z,d_Z), & \quad f_a(x)&\coloneqq(x,\varepsilon/2), \\
            f_b&\colon (X,d_b)\to (Z,d_Z), &  \quad f_b(x)&\coloneqq(x,0),
        \end{aligned}
    \]
    are distance preserving. In addition, the Gromov-Hausdorff distance between $(X,d_a)$ and $(X,d_b)$ is not greater than $\varepsilon/2$.
\end{proposition}

This result is instrumental to prove Gromov-Hausdorff convergence in the following theorem, which will be needed in the proof of \Cref{thm:uniform_convergence}.
\begin{theorem}[{\cite[Theorem~3.2]{perales_monotone_2025}}]
    Let $(X,d_j)$ be a sequence of metric spaces whose diameters are uniformly bounded above and such that the sequences $d_j$ are monotonically increasing. Then, the sequence $d_j$ converges pointwise to a distance $d_\infty$. If, moreover, $(X,d_\infty)$ is compact, we have uniform convergence of the distances and Gromov-Hausdorff convergence of the metric spaces.
\end{theorem}

On another note, we introduce the notion of metric pairs.
\begin{definition}[Metric pair]
    A metric pair is a triple \((X,d,A)\), where \((X,d)\) is a metric space and \(A \subset X\) is a closed subspace endowed with the induced metric by \(X\).
\end{definition}
Next we give the definition of convergence of metric pairs presented in the work of Che, Galaz-García, Guijarro and Membrillo Solis \cite{Che2024}:
\begin{definition}
    Consider a sequence of metric pairs \( \{(X_j, d_j, A_j)\}_{j \in \N}\). We say that the sequence {\em converges in the metric pair sense to} a metric pair \((X_\infty, d_\infty, A_\infty)\) if
    \[
        \lim_{j \to +\infty} \quad 
        \inf \left\{\tilde{d}^H(\phi_j(X_j), \psi_j(X_\infty)) + \tilde{d}^H(\phi_j(A_j), \psi_j(A_\infty)) \right\} = 0,
    \]
    where the infimum is taken over all metric spaces \((Z,\tilde{d})\) and all isometric embeddings \(\phi_j \colon X_j \to Z,\) \(\psi_j \colon X_\infty \to Z\), and \(\tilde{d}^H\) is the Hausdorff distance between subsets of the space \((Z, \tilde{d})\).
\end{definition}
\begin{remark}
    Notice that convergence of metric pairs implies Gromov-Hausdorff convergence of the sequences \((X_j,d_j)\) and \((A_j, {d_j}_{\vert A})\) to the spaces \((X_\infty, d_\infty)\) and \((A_\infty, {d_\infty}_{\vert A_\infty})\), respectively.
\end{remark}

We recall the notion of {\em future developed convergence} given in \cite[Def. 4.10]{sakovich_introducing_2024} for smooth spaces and then extended to timed metric spaces in \cite[Def. 2.8]{perales_timed_hausdoff_2025}, which can be stated in terms of convergence of metric pairs:
\begin{definition}[{\cite[Def. 4.10]{sakovich_introducing_2024}}]\label{def:fd_gh_convergence}
    A sequence of future developed timed metric spaces \((X_j, d_j, \tau_j)\) converges in the {\em future developed sense} to a future developed timed metric space \((X_\infty, d_\infty, \tau_\infty)\) if the sequence of metric pairs \((X_j, d_j, A_j\coloneqq \tau_j^{-1}(0))\) converges in the metric pair sense to the metric pair \((X_\infty, d_\infty, A_\infty \coloneqq \tau_\infty^{-1}(0))\).
\end{definition}

We finally recall that Perales showed that
future developed convergence implies timed-Hausdorff convergence \cite[Theorem 1.5]{perales_timed_hausdoff_2025}.

\section{Null compactifiable cosmological spacetimes}

In this section we show that cosmological spacetimes with compact slices satisfying assumption~\eqref{eq:A} are causally-null compactifiable. More precisely, we show that the metric completion of their associated timed-metric spaces are bi-Lipschitz equivalent to the topological product \([0,\tau_{\max}] \times M\). 

\subsection{Regularity of the cosmological time function}

The first step is to ensure that the cosmological time of cosmological spacetimes with compact slices is regular. Without the compactness assumption the answer is no, even if the slices $(M_t, h|_{M_t})$ are required to be complete: the example featured in the work by Sánchez \cite[Ex.~6.1]{sanchez_globally_2022} shows that this implication does not even hold in the warped product case, for which we know that the null distance is always definite and additionally encodes causality if the slices are complete (cf.~\cite[Lemmata~3.22,3.24]{sormani_null_2016}). The converse implication, namely that the cosmological time being regular implies that the slices are complete is, surprisingly, also not true (see~\cite[Ex.~6.4]{sanchez_globally_2022}). However, when the slices are compact, the cosmological time is regular, as the following lemma shows:

\begin{lemma}\label{lemma:regular_cosmotime_when_compact_slices}
    Let $N=I\times M$ be a cosmological spacetime with compact slices. Then the cosmological time, which is a translation of the coordinate $t$, is regular.
\end{lemma}
\begin{proof}
    The spacetime has $t$ as a time function. Therefore, it is strongly causal (see \cite{minguzzi_lorentzian_2019}). The cosmological time is easily shown to be $\tau_g(t,x)=t-t_1$, where $t_1=\inf I$, so in particular it is finite. So consider a future directed, past-inextendible causal curve $\gamma\colon (s_1,s_2)\to N$ given by $\gamma=(\alpha,\beta)$ with $\alpha$ strictly increasing. Suppose, by contradiction, that $t_1'\coloneqq\inf\{\alpha(s)\mid s\in (s_1,s_2)\}>t_1$. Take some $t^* \in (t'_1,t_2)$. Then $\gamma$ is imprisoned (towards $s_1$) in the compact set $K=[t_1',t^*]\times M$, i.e., there exists some $\delta>0$ such that $\gamma((s_1,s_1+\delta])\subset K$. But having a curve which is inextendible at $s_1$ and imprisoned (towards $s_1$) in a compact set contradicts strong causality. Therefore $t_1'=t_1$. From the expression of the cosmological time one deduces that $\tau_g(\gamma(s))\to 0$ as $s\to s_1^+$.
\end{proof}

\subsection{Compactifiable cosmological spacetimes}\label{sec:compactifiable_cosmo_sptimes}

Without loss of generality, we will always assume that $t_1=0$, i.e., that \(I =(0,\tau_{\max})\) for some value \(\tau_{\max} >0\). Thanks to the previous lemma, applying \Cref{prop:properties-regular-cosmological-time} we can deduce that $\tau_g(t,x)=t$ induces a definite null distance which encodes causality.

We next show that cosmological spacetimes $N=(0,\tau_{\max}) \times M$ with compact slices and satisfying assumption~\eqref{eq:A} are causally-null compactifiable. This reflects the fact that these spaces are thought of as future-developments of an initial data set. We will endow $N$ with the \say{taxi-distance}, defined by
\begin{equation} \label{eq:taxi_distance_def}
    d((t,x),(s,y)) \coloneqq \abs{t-s} +d_{t_0}(x,y),
\end{equation}
where for any choice \(t_0 \in (0,\tau_{\max})\), \(d_{t_0}\) is the distance induced by the Riemannian metric \(h_{t_0}\) on the slice \(\{t_0\} \times M\). 
To do so, we first study the relationship between each $d_t$. 

\subsubsection{Auxiliary results} 

We first state and prove a general lemma for compact metric spaces about the continuity of a function of the form \(x \mapsto \displaystyle\max_{y \in Y} f(x,y)\):

\begin{lemma}\label{lemma:topological_lemma}
    Consider two compact metric spaces \((X,d_X)\) and \((Y,d_Y)\) and their metric product \((X \times Y , d_X + d_Y)\). Let \(f\colon X\times Y \to \R\) be a continuous function and define the function \(\tilde{f} \colon X \to \R\) by
    \[
        \tilde{f}(x) = \min_{y \in Y} f(x,y) \qquad (\text{respectively}, \ \tilde{f}(x) = \max_{y \in Y} f(x,y)).
    \]
    Then \(\tilde{f}\) is continuous.
\end{lemma}
\begin{proof}
    We only prove the minimum case, the maximum one being completely analogous. Firstly notice that, as \(Y\) is compact, \(\tilde{f}\) is well defined. To prove continuity, we fix any \(x \in X\) and any \(\varepsilon>0\) and we look for a sufficiently small \(\delta>0\) such that \(\tilde{f}(B_\delta(x)) \subset B_\varepsilon(\tilde{f}(x))\). As \(f\) is continuous and \(X \times Y\) is compact, it is uniformly continuous as well. Therefore, there exists a \(\delta>0\) such that for any pair of points \((a,b),(c,d) \in X \times Y\) it holds that
    \[
        d_X(a,c) +d_Y(b,d) < \delta \implies \abs{f(a,b)-f(c,d)} < \varepsilon.
    \]
    We claim that \(\delta\) is the desired value. Indeed, for any point \(x' \in X\) with \(d_X(x,x')< \delta\), we can consider a point \(y' \in Y\) such that \(f(x',y') = \tilde{f}(x')\). We then get
    \[
       d_X(x,x') + d_Y(y',y') = d_X(x,x')<\delta \implies \abs{f(x,y') -f(x',y')} < \varepsilon, 
    \]
    which in particular implies
    \[
        \tilde{f}(x') + \varepsilon = f(x',y') + \varepsilon > f(x,y') \geq \tilde{f}(x)
    \]
    If we now consider \(y \in Y\) such that \(f(x,y) = \tilde{f}(x)\), we can repeat the same reasoning using \((x,y)\) and \((x',y)\), obtaining
    \[
        \tilde{f}(x) + \varepsilon = f(x,y) + \varepsilon > f(x',y) \geq \tilde{f}(x').
    \]
    Combining these inequalities we get \(\bigl\vert\tilde{f}(x) - \tilde{f}(x')\bigr\vert < \varepsilon\), which concludes the proof.
\end{proof}
    
It is well known that any two Riemannian metrics on a compact manifold \(M\) are Lipschitz equivalent. In the following results, we employ the previous lemma to show that for cosmological spacetimes $N=I\times M$ with compact slices the Lipschitz constants relating the Riemannian metrics $h_t$ on $M$ can be chosen to depend continuously on $t$.
\begin{lemma}\label{lemma:riemannian_metrics_are_lipschitz_equivalent}
    Consider a cosmological spacetime \((N,g)\) with compact slices. For any \(t_0 \in (0,\tau_{\max})\) there exist continuous functions \(\overline{C},\underline{C} \colon (0,\tau_{\max}) \to (0,\infty)\) such that for any point \((p,v) \in TM\) it holds that
    \[
        \underline{C}(t)\sqrt{h_{t_0}(v,v)_p} \leq \sqrt{h_{t}(v,v)_p} \leq \overline{C}(t)\sqrt{h_{t_0}(v,v)_p}
    \]
    and
    \[
        \lim_{t\to t_0} \underline{C}(t)= \lim_{t\to t_0} \overline{C}(t) = 1.
    \]
    If assumption \eqref{eq:A} holds, the same result is true for any $t_0\in[0,\tau_{\max}]$ and one obtains continuous functions \(\overline{C},\underline{C}\) defined on \( [0,\tau_{\max}]\).
\end{lemma}
\begin{proof}
    Let $0<\varepsilon<\tau_{\max}/2$ and define the function
    \[
        \xymatrix@C=5pt@R=1pt{
            f_\varepsilon \colon & [\varepsilon,\tau_{\max}-\varepsilon] \times STM_{h_{t_0}} \ar[rr] && \ \R \\
             & (t,(p,v)) \ar@{|->}[rr] && \ \sqrt{h_{t}(v,v)_p}
        }
    \]
    where \(STM_{h_{t_0}}\) is the unit tangent bundle of $M$ with respect to the Riemannian metric \(h_{t_0}\). The function \(f_\varepsilon\) is continuous as the metrics \(h_t\) vary smoothly. As \(M\) is compact, it is well known that \(STM_{h_{t_0}}\) is compact as well. Hence, the domain of \(f_\varepsilon\) is the product of two compact manifolds and thus, since any manifold is metrizable, a product of compact metric spaces. As the product metric on the product induces the same topology as the product topology, \(f_\varepsilon\) is continuous even when considered as a function from the metric product of \([\varepsilon,\tau_{\max}-\varepsilon]\) and \(STM_{h_{t_0}}\). Thus, \Cref{lemma:topological_lemma} applies and the functions \(\overline{C}_\varepsilon(t)\) and \(\underline{C}_{\,\varepsilon}(t)\) defined as the maximum and the minimum respectively of the function \(f_\varepsilon(t,(p,v))\) over \(\smash{STM_{h_{t_0}}}\) are continuous. Now, it is clear that if $0<\varepsilon'<\varepsilon$, then $f_{\varepsilon'}|_{[\varepsilon,\tau_{\max}-\varepsilon]}=f_\varepsilon$, and therefore similarly for \(\overline{C}_{\varepsilon'}\) and \(\underline{C}_{\,\varepsilon'}\). As a consequence, there exist continuous functions \(f,\overline{C},\underline{C}\) defined on \((0,\tau_{\max})\) which extend, respectively, all \(f_\varepsilon\), \(\overline{C}_\varepsilon\) and \(\underline{C}_{\,\varepsilon}\).
    
    Notice that the function \(f(t_0,(p,v))\) is constantly equal to 1 for every choice of \((p,v) \in STM_{h_{t_0}}\). In particular we have that \(\overline{C}(t_0) = \underline{C}(t_0)=1\) and thus by continuity we get the second claim of the lemma. About the first one, we notice that it is trivially satisfied whenever \(v=0\), so we will assume that \(v\neq0\). We then have that
    \[
        \sqrt{h_{t}(v,v)_p} = \sqrt{h_{t_0}(v,v)_p} \;
        \sqrt{h_{t}\left(\frac{v}{\sqrt{h_{t_0}(v,v)_p}}, \frac{v}{\sqrt{h_{t_0}(v,v)_p}}\right)_p},
    \]
    so that now the vector is normalized with respect to \(h_{t_0}\) and thus belongs to \(STM_{h_{t_0}}\). We can then estimate from above and below by \(\overline{C}(t)\) and \(\underline{C}(t)\), yielding
    \[
        \underline{C}(t)\sqrt{h_{t_0}(v,v)_p}\leq \sqrt{h_{t}(v,v)_p} \leq \overline{C}(t)\sqrt{h_{t_0}(v,v)_p},
    \]
    which proves the claim. Finally, if the tensors $h_t$ can be continuously extended to $t=0$ and $t=\tau_{\max}$, then one can work directly with the map $f$ defined on \([0,\tau_{\max}]\), obtaining the desired result.
\end{proof}
    As an immediate corollary we get the following estimate on the induced distances \(d_t\) on the slices $M_t$:
\begin{corollary}\label{corollary:distances_are_lipschitz_equivalent}
    Let \((N,g)\) be a cosmological spacetime with compact slices. The induced distances \(d_t\) are all Lipschitz equivalent. In particular we have that
    \[
        \underline{C}(t) d_{t_0} \leq d_t \leq \overline{C}(t) d_{t_0}.
    \]
\end{corollary}

\subsubsection{Proof of \texorpdfstring{\Cref{thm:compactness_nullcompletion}}{compactness of the associated metric space}}
To show that a cosmological spacetime $N=(0,\tau_{\max})\times M$ with compact slices and satisfying assumption~\eqref{eq:A} is causally-null compactifiable, we will prove that the null distance on such a space is bi-Lipschitz to the \say{taxi-distance} \eqref{eq:taxi_distance_def}. To do so, we employ \Cref{lemma:riemannian_metrics_are_lipschitz_equivalent} to give a bound for the null distance between points in the same time-slice in terms of the Riemannian distance induced on the slice. Such a result was first proven in \cite[Theorem 1.2]{nigri2025nulldistancetemporalfunctions} in great generality. Here, we state the theorem and give an alternative proof.
\begin{theorem}\label{thm:generalised_upper_bound}
    Consider a spacetime \((N,g)\) which admits a smooth temporal function \(f \colon N \to \R\), that is \(f\) has past directed timelike gradient. Suppose that there exists \(t_0 \in \R\) such that on the level set \(S=f^{-1}(t_0)\) we have that \(g(\nabla f, \nabla f) = -K^2\) for some constant \(K >0\). We denote by \(h=h_{t_0}\) the Riemannian metric tensor obtained by restricting \(g\) to \(S\). Then for all \(p,q \in S\) it holds that
    \[
        \hat{d}_f(p,q) \leq Kd_{h}(p,q),
    \]
    where \(\d_f\) is the \(f\)-null distance, obtained by replacing the cosmological time in \Cref{def:null-distance} by \(f\) and \(d_h\) is the distance in $S$ induced by the Riemannian tensor $h$.
\end{theorem}
\begin{proof}
    The result is trivial if \(p=q\), so we assume \(p \neq q\). We fix \(\varepsilon' \in (0,1)\) and consider any two points \(p,q \in S\) together with an \(\varepsilon'\)-geodesic in \((S,h)\) joining them, that is a curve \(\sigma \colon [a,b] \to S\) such that \(d_h(p,q) + \varepsilon' > L_h(\sigma)\), parametrized with unit speed.

    The main idea is to build a piecewise lightlike curve joining \(p\) to \(q\) whose \say{spatial component} is given by \(\sigma\). To do so, for every \(t \in [a,b]\) we employ the generalized Gauss Lemma at \(\sigma(t)\), which yields the existence of a neighborhood of \(\sigma(t)\) where, for any choice of local coordinates \(\theta\) in the manifold \(S\), we have that \((f,\theta)\) are local coordinates of \(N\) at \(\sigma(t)\). Moreover, with respect to this chart the metric can be written as
    \[
        g = -\frac{1}{\abs{\nabla f}^2(f,\theta)} \, df \otimes df + g_{ij}(f,\theta) \, d\theta^i \otimes d\theta^j.
    \]
    Notice that in these coordinates, points of the form \((t_0, \theta)\) lie in \(S\) and the Riemannian metric on \(S\) reads as \(h = g_{ij}(t_0,\theta) \, d\theta^i \otimes d\theta^j\).
    
    Up to shrinking, it is not restrictive to assume that the neighborhood of \(\sigma(t)\) is of the form \(U_t \coloneqq [-\delta_t +t_0, \delta_t+t_0] \times B_h(\sigma(t), r_t)\) in these coordinates for some \(\delta_t, r_t > 0\). By compactness of the curve \(\sigma\), there exists a finite set of times \(\{t_l\}_{l=0}^{N_1}\) corresponding to finitely many neighborhoods \(U_l \coloneqq U_{t_l}\) of the above form covering the curve. By taking \(\tilde{\delta} < \min_l \delta_{t_l}\) and shrinking the neighborhoods we can assume that their first factor is given by \([-\tilde\delta +t_0, t_0 + \tilde\delta]\). 
    
    For each \(l \in \{0,\dots,N_1\}\) we can apply the same reasoning in the proof of \Cref{lemma:riemannian_metrics_are_lipschitz_equivalent} to the family of metrics \(g_{ij}(f,\theta) \, d\theta^i \otimes d\theta^j\) on the manifold \([-\tilde\delta +t_0, t_0 + \tilde\delta]\times B_h(\sigma(t_l), r_{t_l})\) to show the existence of continuous functions \(\overline{C_l}(f), \underline{C_l}(f) \colon [-\tilde\delta + t_0, \tilde\delta+ t_0] \to (0,\infty)\) such that
    \[
        \underline{C_l}(f)^2 \, h  \leq g_{ij}(f,\theta) \, d\theta^i \otimes d\theta^j\leq \overline{C_l}(f)^2 h.
    \]

    We define the function \(\overline{C} \colon [-\tilde\delta + t_0, t_0 +\tilde\delta] \to \R\) as \(\overline{C}(t) =\max_l \overline{C_l}(t)\). Notice that this is a continuous function with \(\overline{C}(t_0)=1\).
    
    Fixing \(\varepsilon \in (0,\tilde{\delta})\), since \(\abs{\nabla f}(\sigma(t)) = K\), by continuity we can find product neighborhoods \(V_t\) of points on the curve each contained in some \(U_l\) and such that on each of them the oscillation of \(\abs{\nabla f}\) is bounded by \(\varepsilon\). As these neighborhoods cover \(\sigma([a,b])\), which is compact, we can again extract a finite subcover; moreover we can apply Lebesgue's number Lemma to ensure the existence of a positive \(\delta > 0\) such that for any pair of times \(t_1<t_2\) with \(\abs{t_1-t_2}<\delta\), the image \(\sigma([t_1, t_2])\) is entirely contained in one element of the subcover. Hence we partition the interval \([a,b]\) into segments of length less than \(\delta\) and we work on each segment separately. 
    
    Let \([c,d]\) be any of these segments, so that \(\sigma([c,d]) \subset V_t \subset U_l\) for some \(t \in [a,b]\) and some $l\in\{0,\ldots,N_1\}$.
    We now look for a value \(N \in \N\) such that, defining \(c_k = c+k\cdot (d-c)/N\) for \(k \in \{0, \dots, N-1\}\), on each interval \([c_k, c_{k+1}]\) we have that
    \[
        \int_{c_k}^{c_{k+1}} \abs{\nabla f}(\alpha(s),\theta(\sigma(s)))\sqrt{g_{ij}\bigl(\alpha(s), \theta(\sigma(s))\bigr)\  \dot{\sigma}^i(s) \, \dot{\sigma}^j(s)}\, ds < \min(\varepsilon, \delta_t/2)
    \]
    for any continuous function \(\alpha \colon[c_k, c_{k+1}] \to [-\delta_t/2 +t_0, \delta_t/2 +t_0]\).
    Such an \(N\) exists as, denoting by \(M\) the maximum of the function \(\overline{C}(f)\) over the interval \([-\delta_t +t_0, \delta_t +t_0]\), we can see that
    \[
        \begin{split}
            \int_{c_k}^{c_{k+1}} \abs{\nabla f}(\alpha(s),\theta(\sigma(s)))&\sqrt{g_{ij}\bigl(\alpha(s), \theta(\sigma(s))\bigr)\  \dot{\sigma}^i(s) \, \dot{\sigma}^j(s)}\, ds \\
            &\leq \int_{c_k}^{c_{k+1}} (K + \varepsilon) \, \overline{C}(\alpha(s)) \sqrt{h(\dot{\sigma}(s),\dot{\sigma}(s))} \, ds \\
            &\leq M(K+\tilde{\delta})\int_{c_k}^{c_{k+1}} \sqrt{h(\dot{\sigma}(s),\dot{\sigma}(s))} \, ds \\
            &= M(K+\tilde{\delta})\cdot L_{h}(\sigma_{\vert [c_k,c_{k+1}]})\\
            &=M(K+\tilde{\delta}) \abs{c_k -c_{k+1}} =\frac{M(K+\tilde{\delta})}{N},
        \end{split}
    \]
    where in the first line we used the fact that \(\abs{\nabla f}\) has bounded oscillation inside the neighborhood \([-\delta_t +t_0, \delta_t +t_0] \times B_h(\sigma(t),r_t)\).
    Hence, choosing \(N > M(K+\tilde{\delta})/\min(\varepsilon, \delta_t/2)\) is sufficient.

    We then look at the restriction of \(\sigma\) on each interval \([c_k, c_{k+1}]\); we consider the function \(u \colon [c_k, c_{k+1}] \to \R\) defined as the solution of the following ODE
    \begin{equation}\label{eq:reparametrization_trick_u}
        \begin{dcases}
            u(c_k) = 0\\
            u'(s) = \abs{\nabla f}(u(s)+ t_0,\theta(\sigma(s)))\sqrt{g_{ij}(t_0+u(s), \theta(\sigma(s))) \, \dot{\sigma}^i(s) \, \dot{\sigma}^j(s)}
        \end{dcases}
    \end{equation}
    and the function \(w \colon [c_k, c_{k+1}] \to \R\) defined by
    \begin{equation}\label{eq:reparametrization_trick_w}
        \begin{dcases}
            w(c_{k+1}) = 0\\
            w'(s) = -\abs{\nabla f}(w(s)+t_0,\theta(\sigma(s)))\sqrt{g_{ij}(t_0+w(s), \theta(\sigma(s))) \, \dot{\sigma}^i(s) \, \dot{\sigma}^j(s)}\,.
        \end{dcases}
    \end{equation}
    Such solutions exists by standard ODE theory, as the right hand side is smooth and bounded from above and below thanks to \Cref{lemma:riemannian_metrics_are_lipschitz_equivalent} along with the fact that \(\abs{\nabla f}\) has bounded oscillation. Notice that both functions are strictly monotone with \(u\) being increasing and \(w\) decreasing, since the derivative never vanishes as \(\sigma\) is a geodesic of constant non-vanishing speed. Moreover, by our choice of the intervals \([c_k,c_{k+1}]\) both \(u\) and \(w\) take values in the interval \([0,\min(\varepsilon, \delta_t/2))\). 
    
    Notice that the function \(u-w\) is strictly increasing, with \((u-w)(c_k) = -w(c_k) <0\) and \((u-w)(c_{k+1}) = u(c_{k+1}) > 0\). Hence there exists a unique intermediate time \(b_k \in [c_k,c_{k+1}]\) such that \((u-w)(b_k) = 0\), or equivalently \(u(b_k) = w(b_k)\). We also notice that both functions are bounded above by \(\varepsilon\) by our choice of \(N\).

    We now define a curve \(\beta_k \colon [c_k,c_{k+1}] \to [-\delta_t +t_0, \delta_t +t_0] \times B_h(\sigma(t), r_t)\) as
    \[
        \beta_k(s) = 
        \begin{dcases}
            (t_0 +u(s),\sigma(s)) \ \text{if} \ t \in [c_k,b_k]\\
            (t_0 +w(s), \sigma(s)) \ \text{if} \ t \in [b_k, c_{k+1}],
        \end{dcases}
    \]
    where we recall that the first coordinate function is given by \(f\). By our choice of \(b_k\) such a curve is continuous, starting at \((t_0,\sigma(c_k))\) and ending at \((t_0, \sigma(c_{k+1}))\). Moreover it is immediate to see that this curve is piecewise causal using \eqref{eq:reparametrization_trick_u} and \eqref{eq:reparametrization_trick_w} and in fact piecewise null; the first piece is future directed as  \(u'>0\), whereas the second part is past directed since \(w'<0\). We then have that
    \[
        \begin{split}
            \hat{L}_f(\beta_k) &= u(b_k) - u(c_k) + w(b_k) -w(c_{k+1})\\
                            &= \int_{c_k}^{b_k} u'(s) \,ds -\int_{b_k}^{c_{k+1}} w'(s) \, ds\\
                            &= \int_{c_k}^{b_k}\abs{\nabla f}(u(s)+ t_0,\theta(\sigma(s)))\sqrt{g_{ij}(t_0+u(s), \theta(\sigma(s))) \, \dot{\sigma}^i(s) \, \dot{\sigma}^j(s)} \,ds \\
                            &+ \int_{b_k}^{c_{k+1}}\abs{\nabla f}(w(s)+t_0,\theta(\sigma(s)))\sqrt{g_{ij}(t_0+w(s), \theta(\sigma(s))) \, \dot{\sigma}^i(s) \, \dot{\sigma}^j(s)}ds \\
                            &\leq \int_{c_k}^{b_k}(K+\varepsilon) \, \overline{C}(t_0+u(s)) \sqrt{h(\dot{\sigma}(s),\dot{\sigma}(s))} \,ds \\
                            &\quad + \int_{b_k}^{c_{k+1}}(K+\varepsilon)\, \overline{C}(t_0+w(s)) \sqrt{ h(\dot{\sigma}(s),\dot{\sigma}(s))} \,ds. 
        \end{split}
    \]
    We now recall that both \(u\) and \(w\) take values in \([0,\min(\varepsilon,\delta_t/2))\), thus we can bound the function \(\overline{C}\) in both integrands by its supremum over the interval \([t_0,t_0 +\min(\varepsilon, \delta_t/2))\). Therefore we get that
    \[
        \begin{split}
            \hat{L}_f(\beta_k) &\leq
            (K+\varepsilon) \int_{c_k}^{c_{k+1}} \sqrt{h(\dot{\sigma}(s),\dot{\sigma}(s))} \, ds\:
            \sup_{s' \in [t_0, t_0+\min(\varepsilon, \delta_t/2))} \overline{C}(s')
            \\ 
            &\leq  (K+\varepsilon) L_{h}(\sigma_{\vert [c_k,c_{k+1}]})\sup_{s' \in [t_0, t_0+\varepsilon)} \overline{C}(s').
        \end{split}
    \]
    For each \(k \in \{0,\dots, N-1\}\) we can build such a curve \(\beta_k\) and each of them can be concatenated, since the endpoint of \(\beta_k\) is the starting point of \(\beta_{k+1}\). We then get a piecewise causal curve \(\beta_{c,d}\) from \(\sigma(c)\) to \(\sigma(d)\) such that
    \[
        \begin{split}
            \hat{L}_f(\beta_{c,d}) &= \sum_{k=0}^{N-1} \hat{L}_f(\beta_k) \leq (K+\varepsilon) \sum_{k=0}^{N-1} L_h(\sigma_{\vert [c_k, c_{k+1}]})\sup_{s' \in [t_0, t_0+\varepsilon)}\overline{C}(s') \\
            &= (K+\varepsilon) L_h(\sigma_{\vert[c,d]})\sup_{s' \in [t_0, t_0+\varepsilon)}\overline{C}(s').
        \end{split}
    \]
    As this holds on every interval \([c,d]\) of the original partition of \([a,b]\), we can again concatenate all the curves \(\beta_{c,d}\) to get a piecewise causal curve \(\beta\) joining \(p\) to \(q\) which, by additivity of \(\hat{L}_f\) and \(L_h\) will satisfy
    \[
        \d_f(p,q) \leq\hat{L}_f(\beta) \leq (K+\varepsilon)L_h(\sigma) \sup_{s' \in [t_0, t_0+\varepsilon)}\overline{C}(s')  \leq (K+\varepsilon) (d_h(p,q) + \varepsilon') \sup_{s' \in [t_0, t_0+\varepsilon)}\overline{C}(s').
    \]
    Sending \(\varepsilon\to0\) and then \(\varepsilon' \to 0\), and using continuity of the function \(\overline{C}\) we get that
    \[
        \hat{d}_f(p,q) \leq K d_h(p,q). \qedhere
    \]
\end{proof}

\begin{corollary}\label{cor:bound_nulldist_basedist}
    For a cosmological spacetime $N=(0,\tau_{\max})\times M$ and $t_0\in (0,\tau_{\max})$, we have that
    \[
        \d_g((t_0,x),(t_0,y)) \leq d_{t_0}(x,y).
    \]
\end{corollary}
\begin{proof}
    The cosmological time function \(t\) is a temporal function whose gradient has constant norm \(-1\), hence applying \Cref{thm:generalised_upper_bound} to the level set \(\{t=t_0\}\) immediately yields the claim.
\end{proof}
\begin{lemma}\label{lem:bound_null_distance}
    Let $N=(0,\tau_{\max})\times M$ be a cosmological spacetime, $s,t\in(0,\tau_{\max})$ and $x,y\in M$. Fix $t_0\in(0,\tau_{\max})$. Then
    \[
    \hat{d}_g\bigl((t,x),(s,y)\bigr)\leq \abs{t-t_0}+\abs{s-t_0}+d_{t_0}(x,y).
    \]
    In particular, if there exists some $t_0\in(0,\tau_{\max})$ such that $\diam_{d_{t_0}}(M)$ is bounded above by $D$, then $\diam_{\hat{d}_g}(\bar{N})$ is bounded above by $2\tau_{\max}+D$.
\end{lemma}
\begin{proof}
    Consider the points $(t,x)$, $(t_0,x)$, $(t_0,y)$ and $(s,y)$ and use the triangle inequality for $\d_g$ and \Cref{cor:bound_nulldist_basedist}.
\end{proof}
 The next result gives a Lipschitz bound from above for the null distance in terms of the taxi distance defined in \eqref{eq:taxi_distance_def}.
\begin{corollary}\label{corollary:lip_embedding_above}
    Let $N=(0,\tau_{\max})\times M$ be a cosmological spacetime with compact slices satisfying assumption~\eqref{eq:A}. Fix \(t_0 \in [0,\tau_{\max}]\) and define \(C = \max_{t\in[0,\tau_{\max}]} \overline{C}(t)\), where \(\overline{C}(t)\) is defined as in \Cref{lemma:riemannian_metrics_are_lipschitz_equivalent}. For any pair of points \((t,x),(s,y) \in N\) it holds that
    \[
        \d_g((t,x),(s,y)) \leq \max(1,C)(\abs{t-s} +d_{t_0}(x,y)).
    \]
\end{corollary}
\begin{proof}
    Applying triangle inequality yields that
    \[
        \begin{split}
            \d_g((t,x),(s,y)) &\leq \d_g((t,x),(t,y)) +\d_g((t,y),(s,y))\\
                                 &= \d_g((t,x),(t,y)) + \abs{t-s}\\
                                 &\leq d_t(x,y) + \abs{t-s},
        \end{split}
    \]
    where in the second line we used the fact that \((t,y)\) and \((s,y)\) must be causally related and thus their null distance is equal in absolute value to the difference of their time coordinate and in the third line we used \Cref{cor:bound_nulldist_basedist}. We now employ \Cref{corollary:distances_are_lipschitz_equivalent} to get
    \[
        \begin{split}
            d_t(x,y) + \abs{t-s} &\leq \overline{C}(t)d_{t_0}(x,y) + \abs{t-s}\\
                                 &\leq Cd_{t_0}(x,y) + \abs{t-s}\\
                                 &\leq \max(1,C)(d_{t_0}(x,y) + \abs{t-s}),
        \end{split}
    \]
    which concludes the proof.
\end{proof}
We now work our way towards a lower bound in terms of the taxi distance. The first step in this direction is given by the following:
\begin{proposition}
    Let $N=(0,\tau_{\max})\times M$ be a cosmological spacetime with compact slices satisfying assumption~\eqref{eq:A}. Fix \(t_0 \in [0,\tau_{\max}]\) and denote \(c=\min_{[0,\tau_{\max}]}\underline{C}(t)\), where \(\underline{C}(t)\) is defined as in \Cref{lemma:riemannian_metrics_are_lipschitz_equivalent}. For any pair of points \((t,x),(s,y) \in N\) it holds that
    \[
        \hat{d}_g((t,x),(s,y)) \geq c\cdot d_{t_0}(x,y).
    \]
\end{proposition}
\begin{proof}
    Consider any piecewise causal curve \(\beta = \beta_1 \dots \beta_k\) between the two points and orient each piece \(\beta_i\) so that it is future directed. We can assume then that \(\beta_i \colon [t_1^i, t_2^i] \to N\) with \(t_1^i < t_2^i\) is of the form \(\beta_i(t) =(t,\sigma_i(t))\). We have that
    \[
        \begin{split}
            \hat{L}(\beta) &= \sum_{i=1}^{k} t_2^i -t_1^i\\
                           &\geq \sum_{i=1}^k \int_{t_1^i}^{t_2^i} \sqrt{h_{t}(\dot{\sigma_i}(t) , \dot{\sigma_i}(t))} \, dt\\
                           & \geq \sum_{i=1}^k \int_{t_1^i}^{t_2^i} \underline{C}(t) \sqrt{h_{t_0}(\dot{\sigma_i}(t) , \dot{\sigma_i}(t))} \, dt\\
                           & \geq c\sum_{i=1}^k \int_{t_1^i}^{t_2^i} \sqrt{h_{t_0}(\dot{\sigma_i}(t) , \dot{\sigma_i}(t))} \, dt = c \cdot L_{t_0}(\sigma),
        \end{split}
    \]
    where \(\sigma\) is just the projection of the curve \(\beta\) on the factor \(M\), given by the concatenation of all the curves \(\sigma_i\) with appropriate orientation. Notice that \(\sigma\) joins \(x\) to \(y\) and has therefore length greater or equal to \(d_{t_0}(x,y)\). Hence we found that
    \[
        \hat{L}(\beta) \geq c \cdot d_{t_0}(x,y),
    \]
    thus taking the infimum over all piecewise causal curves joining \((t,x)\) to \((s,y)\) yields
    \[
        \hat{d}_g((t,x),(s,y)) \geq c \cdot d_{t_0}(x,y),
    \]
    concluding the proof.
\end{proof}
\begin{corollary}\label{corollary:lip_embedding_below}
    Let $N=(0,\tau_{\max})\times M$ be a cosmological spacetime with compact slices satisfying assumption~\eqref{eq:A} and fix \(t_0 \in [0,\tau_{\max}]\). For any pair of points \((t,x),(s,y) \in N\) we have that
    \[
        \hat{d}_g((t,x),(s,y)) \geq \min\left( \frac{1}{2}, \frac{c}{2} \right) (d_{t_0}(x,y) + \abs{s-t}).
    \]
\end{corollary}
\begin{proof}
    The previous proof showed that \(\hat{d}_g((t,x),(s,y)) \geq c \cdot d_{t_0}(x,y)\). On the other hand, by the properties of the null distance, we know that
    \(\hat{d}_g((t,x),(s,y)) \geq \abs{s-t}\). Summing these two inequalities and dividing by 2 yields
    \[
        \hat{d}_g((t,x),(s,y)) \geq \frac{1}{2}\abs{s-t} + \frac{c}{2} \cdot d_{t_0}(x,y).
    \]
    Taking the minimum value between \(1/2\) and \(c/2\) we can estimate from below and prove the claim.
\end{proof}

We now have all the needed tools to prove \Cref{thm:compactness_nullcompletion}:
\begin{proof}[Proof of {\Cref{thm:compactness_nullcompletion}}]
    Consider the map
    \[
        \xymatrix@R=2pt@C=1pt{
            i\colon & (N,\hat{d}_g) \ar[rr] && ([0,\tau_{\max}] \times M, d)\\
            & (t,x) \ar@{|->}[rr] && (t,x)
        }
    \]
    Then \Cref{corollary:lip_embedding_above} and \Cref{corollary:lip_embedding_below} ensure that this map is a bi-Lipschitz embedding; in particular, the map \(i\) extends to a bi-Lipschitz map \(j \colon (\bar{N}, \d_g) \to ([0,\tau_{\max}] \times M,d)\). Notice that the space \(([0,\tau_{\max}] \times M, d)\) is compact as \(d\) induces the product topology and that \(j\) is an homeomorphism onto its image with the property that \(j(\bar{N})\) is closed. This implies that \(j(\bar{N})\) is compact and therefore so is \(\bar{N}\).

    To conclude the proof, we notice that \(i(N)=N\) is dense in \(([0,\tau_{\max}] \times M,d)\) and recalling that \(j\) extends \(i\) to the completion \(\bar{N}\) in such a way that \(j(\bar{N})\) is closed we see that 
    \[
        [0,\tau_{\max}] \times M = \operatorname{cl}_d(N) \subset \operatorname{cl}_d(j(\bar{N})) = j(\bar{N}).
    \]
    Hence \(j\) is surjective and thus a bi-Lipschitz homeomorphism, which proves the second claim.
\end{proof}

As a consequence of \Cref{lemma:regular_cosmotime_when_compact_slices}, and \Cref{thm:compactness_nullcompletion}, we have that cosmological spacetimes with compact slices and satisfying assumption~\eqref{eq:A}, are causally-null compactifiable.

\section{Monotonic convergence of null distances} \label{sect:monotonic_convergence}

We consider a sequence of cosmological spacetimes as in \Cref{def:generalized_product} having the same base space, that is, we consider the manifold \(N =(0,\tau_{\max}) \times M\) and a sequence of Lorentzian metrics defined as
\[
    g_j = -dt^2 +h_j(t,x),
\]
where each tensor \(h_j(t,x)\) satisfies the assumptions described in \Cref{def:generalized_product}. We will suppress the dependency on \(t,x\) to ease the notation.

Our main goal is to give some conditions on the metric tensors \(h_j\) in order to ensure that the induced null distances are monotonically increasing in \(j\) and converging to a distance. We then show that the sequence of Lorentzian manifolds we are considering actually converges in the future developed sense (cf. \Cref{def:fd_gh_convergence}) to the limit space. This is analogous to what is done in \cite{perales_monotone_2025} for the Riemannian case, where the authors study sequences of Riemannian manifolds whose distances are monotonically converging to a limit distance and show that the convergence, under compactness assumptions, holds both in the uniform and Gromov-Hausdorff sense.

\subsection{Future developed convergence and proof of \texorpdfstring{\Cref{thm:future_developed_convergence}}{FD-convergence}}

We show that a generic spacetime with regular cosmological time and its completion always satisfy the initial data property up to replacing the minimum in the definition with an infimum:
\begin{lemma}\label{lem:initial_data_property}
    Let $(N,g)$ be a spacetime with regular cosmological time. Consider its associated timed-metric space $(\bar{N},\d_g,\tau_g)$ and call $M=\tau_g^{-1}(0)$ its \emph{initial level set}. Then we have 
    \[
        \tau_g(p)= \inf \bigl\{ \d_g(p,q)\mid q\in M\bigr\},\qquad \forall p\in \bar{N}.
    \]
    Moreover, for every point \(p \in N\), the infimum is actually attained. If in addition $M$ is compact, the same is true for all $p\in\bar{N}$.  
\end{lemma}
\begin{proof}
    First of all, take $p\in \bar{N}$, $q\in M$ and consider sequences $\{p_n\},\{q_n\}$ in $N$ converging, respectively, to $p$ and $q$. By the properties of the null distance, $\abs{\tau_g(p_n)-\tau_g(q_n)} \leq \d_g(p_n,q_n)$. Taking limits in both sides we get $\tau_g(p)\leq \d_g(p,q)$, therefore $\tau_g(p)\leq \inf\{ \d_g(p,q) : q\in M \}$.

    Now take a point $p\in N$ and consider one \textit{generator} $\gamma_p$ as in \cite[Theorem~1.2]{andersson_cosmological_1998}, i.e., a future directed timelike curve $\gamma_p\colon (0,\tau_g(p)]\to N$ such that $\gamma_p(\tau_g(p))=p$ and $\tau_g(\gamma_p(s))=s$, for every $s\in (0,\tau_g(p)]$. Consider the extension of such a generator to $M$ (which exists because $\gamma_p$ is $1$-Lipschitz with respect to $\d_g$). Its initial point is $q\coloneqq\gamma_p(0)\in M$. Now consider a sequence $\{q_n\}$ of points in the image of $\gamma_p$ such that $\tau_g(q_n)\to 0$ as $n\to \infty$. One has 
    \[
        \d_g(p,q)=\lim_{n\to\infty} \d_g(p,q_n)= \lim_{n\to\infty} \bigl( \tau_g(p)-\tau_g(q_n) \bigr) =\tau_g(p),
    \]
    which proves the other inequality for every \(p \in N\), with the infimum being attained by \(q\). If \(p \in \bar{N}\), we consider a sequence of points \(\{p_n\}\) in \(N\) converging to \(p\). For each $p_n$, we know there is a point \(q_n \in M\) which satisfies \(\d_g(p_n,q_n) = \tau_g(p_n)\). We therefore have
    \[
        \inf \bigl\{ \d_g(p,q)\mid q\in M\bigr\} \leq \d_g(p,q_n) \leq \d_g(p,p_n) + \d_g(p_n,q_n) = \d_g(p,p_n) + \tau_g(p_n).
    \]
    Taking the limit as \(n \to +\infty\) we see that the right hand side converges to \(\tau_g(p)\), thus proving the missing inequality for all points \(p \in \bar{N}\).
    
    Finally, assume that $M$ is compact and consider $p\in \bar{N}$. Take a sequence $\{p_n\}$ in $N$ converging to $p$. For each $n\in\N$ we have, by the previous step, that $\tau_g(p_n)=\d_g(p_n,r_n)$ for some $r_n\in M$. By the compactness of $M$, the sequence $r_n$ admits a converging subsequence, which we denote the same for simplicity, to some $q\in M$. Therefore, 
    \[
        \tau_g(p)= \lim_{n \to \infty} \tau_g(p_n)=\lim_{n \to \infty} \d_g(p_n,r_n) = \d_g(p,q).
    \]    
    So for every $p\in \bar{N}$ we have that $\tau_g(p)\geq \inf\{ \d_g(p,q) : q\in M \}$, and the opposite inequality was already proved. Moreover, in both cases the infimum is achieved, so it is a minimum.    
\end{proof}

As an immediate corollary, we deduce that every causally-null compactifiable spacetime and every cosmological spacetime with compact slices satisfying assumption~\eqref{eq:A} has the distance from initial data property as in \Cref{def:future-developed}. Equivalently, the timed metric space \((\bar{N}, \d_g, \tau_g)\) is future developed according to \Cref{def:timed_metric_space_future_developed}. Moreover, it is homeomorphic to the topological product \([0,\tau_{\max}] \times M\). Now, we prove future developed convergence (cf.~\Cref{def:fd_gh_convergence}) when considering a sequence of Lorentzian metrics on the same space which induce an increasing sequence of null distances.

In the setting of \cite[Theorem~3.2]{perales_monotone_2025} one actually gets uniform and Gromov-Hausdorff convergence of the sequence to the limit space. To apply such a result to our sequence of cosmological spacetimes, we will show that a sufficient condition for the null distances to be increasing is given by having the sequence of Riemannian tensors on the spatial slices increase with \(j \in \N\). To get convergence to a compact metric space, we moreover assume  a uniform tensor bound on the Riemannian metrics.

\begin{theorem}\label{thm:uniform_convergence}
    For each $j\in \N$, let $N=(0,\tau_{\max})\times M$ be a cosmological spacetime with compact slices and Lorentzian metric $g_j=-dt^2+h_j$. Denote by $h_{j,t}$ the restriction of $h_j$ to the level sets of~$t$. Assume that the sequence satisfies assumption \eqref{eq:B}. Then the sequence of null distances $\d_j$ in the metric completions $(\bar{N}_j,\d_j)$ converges pointwise to a distance $d_\infty$. Moreover, $(\bar{N},d_\infty)$ is compact, so the convergence is uniform.
\end{theorem}
\begin{proof}
    The sequence of tensors $h_{j,t}\coloneqq h_j|_{M_t}$ being increasing with $j$ implies that $j$-causal vectors are ($j-1$)-causal, i.e., the causal cones for $g_j$ are contained in the causal cones for $g_{j-1}$. As a consequence, the sequence $\d_j$ of corresponding null distances is also increasing (see \cite[Lemma~4.15]{allen_properties_2022}). Similarly, the null distance $\d_{\tilde{g}}$ induced by $\smash{\tilde{g}\coloneqq -dt^2+\tilde{h}}$ is greater than each of the $\d_j$'s. As a consequence, $\d_j$ converges pointwise to some function $d_\infty$ and $\d_j\leq d_\infty\leq \d_{\tilde{g}}$. The fact that $d_\infty$ is a distance can be shown using the fact that the $\d_j$'s are distances and that the diameters $\diam_{\d_j}(N)$ are uniformly bounded \cite[Lemma~2.5]{perales_monotone_2025}. Indeed, $\diam_{\d_j}(N)\leq \tau_{\max}+\diam_{d_{\tilde{h}}}(M)$, which is finite.

    Now, \Cref{thm:compactness_nullcompletion} ensures that the completions of $(N,\d_{\tilde{g}})$ and $(N,\d_{1})$ are both homeomorphic to $[0,\tau_{\max}]\times M$ with the product topology. As $\d_1\leq d_\infty\leq \d_{\tilde{g}}$, the topology induced by $d_\infty$ is equivalent to the product topology in $[0,\tau_{\max}]\times M$ and, in particular, $(\bar{N},d_\infty)$ is compact. By \cite[Theorem~3.2]{perales_monotone_2025} we then have uniform convergence.
\end{proof}

One can extend the previous result by proving that the sequence converges in the future developed sense to the final space, which we do in the following proposition:

\begin{proposition}\label{prop:metricPairConv}
    Let $N_j$ be a sequence of cosmological spacetimes as in \Cref{thm:uniform_convergence}. Consider the sequence of associated metric spaces \((\bar{N}_j, \d_j)\) along with their initial level sets \(M = t^{-1}(0)\). Then the sequence \((\bar{N}_j, \d_j, M)\) converges in the metric pair sense to the metric pair \((N_\infty, d_\infty, M)\).
\end{proposition}
\begin{proof}
    We know that we have monotonic uniform convergence of the distances \(\hat{d}_j\) to \(d_\infty\). In particular, for any \(\varepsilon>0\) we have that for \(j\) big enough it holds that
    \[
        d_\infty -\varepsilon \leq \hat{d}_j \leq d_\infty.
    \]
    We fix such a \(j \in \N\). Employing \Cref{prop:embedding_taxi} we consider the space \(Z =([0,\tau_{\max}]\times M \times [0, \varepsilon/2], d_Z)\) with isometric embeddings \(f_j\colon \bar{N}_j\to Z\) and \(f_\infty\colon N_\infty\to Z\) given, respectively, by \(f_j(x) =(x,\varepsilon/2)\) and \(f_\infty(x) = (x,0)\) where, by \Cref{thm:compactness_nullcompletion}, we identify notationally the sets $\bar{N}_j=N_\infty=[0,\tau_{\max}]\times M$. As a consequence we have
    \[
        d_Z^H(f_j(\bar{N}_j), f_\infty(N_\infty)) \leq \varepsilon/2.
    \]
    Recall that the distance \(d_Z\) is built in such a way that \(d_Z((x,s),(y,t)) \leq d_\infty(x,y) + \abs{s-t}\). In particular, we see that if we restrict the embeddings \(f_j\) and \(f_\infty\) to \((\{0\}\times M, \hat{d}_{j \vert M})\) and \((\{0\} \times M, d_{\infty \vert M})\), respectively, we get that
    \[
        d_Z(f_j(x), f_\infty(x)) \leq d_\infty(x,x) + \varepsilon/2 = \varepsilon/2 \quad \forall x \in M,
    \]
    hence
    \[
        d_Z^H(f_j(M), f_\infty(M)) \leq \varepsilon/2.
    \]
    This implies that
    \[
        d_Z^H(f_j(\bar{N}_j), f_\infty(N_\infty)) + d_Z^H(f_j(M), f_\infty(M)) \leq \varepsilon,
    \]
    therefore estimating the metric pair distance between \((\bar{N}_j, \hat{d}_j, M)\) and \((N_\infty, d_\infty, M)\) by \(\varepsilon\). Hence we get convergence of the sequence as \(j \to +\infty\).
\end{proof}

\begin{theorem}\label{thm:initial_data_property_limit}
    Let $(N,g_j)$ be a sequence of cosmological spacetimes with compact slice and fixed underlying manifold $N=(0,\tau_{\max})\times M$. Assume, moreover, that the sequence $\d_j$ of associated null distances is increasing and their extensions to the metric completions are converging uniformly to a distance $d_\infty$. Then the function \(t\) satisfies
    \[
     t(p) = \min \bigl\{ d_\infty(p,q) \mid \ q \in \{0\}\times M \bigr\},\qquad \forall p\in \bar{N}.
    \]
\end{theorem}
\begin{proof}
    Uniform convergence implies that the metric $d_\infty$ induces the same topology on $\bar{N}$ as any of the~$\d_j$. By \Cref{lem:initial_data_property}, we know that
    \[
     t(p) = \min \Big\{ \d_j(p,q) \mid \ q \in \{0\}\times M \Big\},\qquad \forall p\in \bar{N}, \forall j\in \N.
    \]
    Consider $p\in N$. For each $j\in\N$, let $q_j\in \{0\}\times M$ be points attaining the minimum. Consider a $d_\infty$-converging subsequence, which we denote the same for simplicity and let $q$ be its limit. Then 
    \[
    t(p)=\d_j(p,q_j)=\lim_{j\to \infty} \d_j(p,q_j)= d_\infty(p,q),
    \]
    where the last inequality follows by uniform convergence. As a consequence, we have that $t(p)\geq \min\{d_\infty(p,r) \mid r\in \{0\}\times M\}$.
    
    On the other hand, the function \(t\) is \(1\)-Lipschitz for all metrics \(\d_j\) and thus for \(d_\infty\) as well. If we fix a point $r\in \{0\}\times M$ we have that \(t(r) = 0\) and thus
    \[
        d_\infty(p,r) \geq \abs{t(p) -t(r)} = t(p).
    \]
    Taking the infimum over all \(r \in \{0\}\times M\) yields the missing inequality.
\end{proof}
\begin{proof}[Proof of {\Cref{thm:future_developed_convergence}}]
    From \Cref{thm:uniform_convergence} we get that the sequence converges uniformly to the compact limit space \((\bar{N},d_\infty)\). Using \Cref{thm:initial_data_property_limit} we get that the limit space is indeed future developed according to \Cref{def:future-developed}. Thus, by \Cref{prop:metricPairConv}
    we get that the sequence converges in the future developed sense as in \Cref{def:fd_gh_convergence}. We can then employ \cite[Theorem~1.5]{perales_timed_hausdoff_2025} to conclude that the sequence converges in the intrinsic timed-Hausdorff sense as in \cite[Definition~4.22]{sakovich_introducing_2024}.
\end{proof}

\subsection{Causally-null limit spaces}
We now investigate the relation between the limit distance from \Cref{thm:future_developed_convergence} and the notion of {\em causally-null distance} associated to a timed metric space presented in \cite[Def. 5.1]{che_perales_2026}. We briefly recall the definition:

\begin{definition}\label{def:causally_null_distance}
    Given a timed metric space $(X,d,\tau)$, we define the relation \(\leq_{d,\tau}\) as follows:
    \[
        p \leq_{d,\tau} q \iff d(p,q) = \tau(q) - \tau(p),
    \]
    which defines a causal order (cf. \cite[Proposition~4.1]{che_perales_2026}). We then define the {\em causally-null distance} $\hat{d}_{d,\tau}$ by
    \[
        \hat{d}_{d,\tau}(p,q)=\inf \sum_{i=1}^{N} |\tau(p_i)-\tau(p_{i-1})| 
    \]
    where the infimum runs over all finite collections of points $\{p_i\}_{i=0}^{N}$, where $p_{0}\coloneqq p$ and $p_{N} \coloneqq q$, and $p_i$, $p_{i-1}$ are causally related, possibly alternating past with future:
    \[
        p_i \leq_{d,\tau} p_{i-1} \quad \text{ or } \quad p_i \geq_{d,\tau} p_{i-1}, \quad i\in\{1,\dots,N\}
    \]
    In \cite[Proposition~1.9]{che_perales_2026} it is shown that $\d_{d,\tau}$ is a pseudo-distance which is always greater than or equal to \(d\). We say that $(X,d,\tau)$ is {\em causally-null} if  $\hat{d}_{d,\tau}=d$.
\end{definition}
\subsubsection{Proof of \texorpdfstring{\Cref{thm:cosmological_spacetimes_are_causally_null}}{completions of cosmological spacetimes being causally-null}}
All spacetimes with regular cosmological time functions have associated metric spaces which are causally-null (cf. \cite[Proposition~5.6]{che_perales_2026}). However, if we consider the (compact) completion \((\bar{N},\d)\) of the associated metric space of a causally-null compactifiable spacetime \((N,g)\) and endow it with the causal relation given in \Cref{def:causally_null_distance}, it is not always true that it is causally-null (see \Cref{ex:triangle}). However, with some additional assumption of \say{causal connectivity} of the completed space we can guarantee that the completion is causally-null: 
\begin{proposition}\label{prop:causally_null_completions}
    Consider a timed metric space \((N, \d,\tau)\). Suppose that the completion \((\bar{N},\d)\) endowed with the causal relation as in \Cref{def:causally_null_distance} satisfies the following assumption: for every point \(p \in \bar{N}\), all of its neighborhoods contain a different point \(q \in N\) such that \(p \leq_{d,\tau} q\) or \(q \leq_{d,\tau} p\). Then the space is causally-null.
\end{proposition}

\begin{proof}
    From \cite[Proposition~1.9]{che_perales_2026} we only need to prove the opposite inequality, i.e., that $\d\geq \d_{d,\tau}$. To this end, let \(p\) and \(q\) be two points in \(\bar{N}\). By our assumption, we can replace them by two points \(p'\) and \(q'\) which are arbitrarily close to the original points and causally related to them. As the points are now contained in the original spacetime, we know that for any \(\varepsilon >0\) there exists a chain of points \(\{q_i\}_{i=1}^m\) such that, setting \(q_0 \coloneqq p'\) and \(q_{m+1} \coloneqq q'\), we have \(q_i \leq q_{i+1}\) or \(q_{i+1} \leq q_i\) for all \(i\in\{0, \dots ,m\}\) and
    \[
        \sum_{i=0}^{m} \abs{\tau(q_i) -\tau(q_{i+1})} \leq \d(p',q') + \varepsilon.
    \]
    Since points that are causally related with respect to \(\leq\) will also be causally related with respect to \(\leq_{d,\tau}\) and we took the points \(p'\) and \(q'\) to be in relation with \(p\) and \(q\) respectively, adding the points \(p\) and \(q\) to the chain will yield a valid chain to compute the \(\d_{\d, \tau}\)-distance:
    \[
        \begin{split}
            \hat{d}_{\d,\tau}(p,q) &\leq \abs{\tau(p) - \tau(p')} + \sum_{i=0}^{m} \abs{\tau(q_i) -\tau(q_{i+1})}+ \abs{\tau(q') - \tau(q)} \\
            & \leq \abs{\tau(p) - \tau(p')} +  \d(p',q') + \varepsilon + \abs{\tau(q) - \tau(q')}.
        \end{split}
    \]
    Sending \(\varepsilon \to 0^+\) we get
    \[
        \hat{d}_{\d,\tau}(p,q) \leq \abs{\tau(p) - \tau(p')} +  \d(p',q') +  \abs{\tau(q) - \tau(q')}.
    \]
    By continuity we can then send \(p' \to p\) and \(q' \to q\) getting \(\hat{d}_{\d,\tau}(p,q) \leq \d(p,q)\), which was the missing inequality.
\end{proof}
The previous result allows us to immediately prove \Cref{thm:cosmological_spacetimes_are_causally_null}:
\begin{proof}[Proof of {\Cref{thm:cosmological_spacetimes_are_causally_null}}]
    By the previous proposition, it is sufficient to show that each point of the completion is the limit of a sequence of points in the interior which are \(\leq_{d,\tau}\)-related to it. As we know from \Cref{thm:compactness_nullcompletion}, we can identify the completion with the product space \([0,\tau_{\max}] \times M\), where the added points coincide with the set \(\{0,\tau_{\max}\} \times M\). Clearly, if a point belongs to the interior the property is true. Let us assume that the point is of the form \((0,x)\) for some \(x \in M\); the case \((\tau_{\max},x)\) is completely analogous. We can consider the sequence of points \(\{(\tau_{\max}/n, x)\}_{n \geq 2}\), which always lies in the original space and see that it clearly converges to \((0,x)\) as \(n \to +\infty\). Fixing any \(n \in \N\), we build the sequence \(\{(\tau_{\max}/kn, x)\}_{k \geq 1}\) and notice that for every \(k \geq 1\) we have
    \[
        \left(\frac{\tau_{\max}}{kn},x\right) \leq \left(\frac{\tau_{\max}}{n},x\right)  \implies \d \left( \left(\frac{\tau_{\max}}{kn},x\right) ,\left(\frac{\tau_{\max}}{n},x\right)\right) = \frac{\tau_{\max}}{n} \left(\frac{k-1}{k}\right).
    \]
    Taking the limit as \(k \to +\infty\) gives
    \[
        \d \left( \left(0,x\right), \left(\frac{\tau_{\max}}{n},x\right)\right) = \frac{\tau_{\max}}{n} \implies (0,x) \leq_{d,\tau} \left( \frac{\tau_{\max}}{n},x \right),
    \]
    which shows that the points of the sequence are \(\leq_{d,\tau}\)-related to \((0,x)\). Hence the previous proposition applies and the completion is causally-null.
\end{proof}

\subsubsection{Proof of \texorpdfstring{\Cref{thm:causally_null_limit_is_null_distance_of_limit_tensor}}{causally-null distance coinciding with the null distance of the limit tensor}}
As a first step towards proving \Cref{thm:causally_null_limit_is_null_distance_of_limit_tensor}, in the following lemma we prove that the causal relation \(\leq_{d,\tau}\) is actually the extension of the causal relation \(\leq\) of the spacetime \((0,\tau_{\max}) \times M\) to the spacetime with boundary \([0,\tau_{\max}] \times M\):
\begin{lemma}\label{lemma:causal_relation_completion_implies_causal_curves}
    Consider the metric completion \(([0,\tau_{\max}]\times M, \hat{d})\) of a cosmological spacetime with compact slices and satisfying assumption \eqref{eq:A}. Then, for any two points \(p=(t_p,x_p), q=(t_q,x_q)\) we have that 
    \[
        \hat{d}(p,q) = \tau(q) - \tau(p) \iff p=q \ \text{or} \ \exists  \ \gamma \ \text{f.d. causal curve from} \ p \ \text{to} \ q,
    \]
    where by causal curve we mean a $C^0$ curve which is causal on the interior of its domain.
\end{lemma}
\begin{proof}
    The implication \say{\( \impliedby\)} is immediate: if \(p\) and \(q\) are related by a future directed causal curve \(\gamma\), then the curve is contained in the interior of the manifold except at most for its endpoints. For any two points \(\gamma(t)\) and \(\gamma(s)\) along the curve with \(t < s\) we know that \(\hat{d}(\gamma(t),\gamma(s)) = \tau(\gamma(s)) - \tau(\gamma(t))\) and thus letting the points approach the endpoints we get the conclusion.

    Conversely, suppose that \(\hat{d}(p,q) = t_q - t_p\). Notice that in particular this implies that \(0\leq t_p \leq t_q \leq \tau_{\max}\) and  moreover $t_p=t_q$ if and only if \(p = q\). We therefore assume that \(t_p < t_q\). We consider the pair of points \(p_\varepsilon = (t_p+\varepsilon,x_p)\) and \(q_\varepsilon = (t_q - \varepsilon, x_q)\) with \(\varepsilon \in (0, (t_p+t_q)/2)\) and notice that we have
    \[
        \d (p,p_\varepsilon) = \tau(p_\varepsilon) - \tau(p) = \varepsilon = \tau(q) - \tau(q_\varepsilon) = \d (q,q_\varepsilon),
    \]
    as shown in \Cref{thm:cosmological_spacetimes_are_causally_null}. Clearly then \(p_\varepsilon\) and \(q_\varepsilon\) belong to the original spacetime \((0,\tau_{\max}) \times M\), which is causally-null. This means that we can find a piecewise causal curve \(\tilde{\gamma}_\varepsilon\) joining the points \(p_\varepsilon\) and \(q_\varepsilon\) such that, on denoting by \(\{p_i^\varepsilon\}_i\) its breakpoints, we have that
    \[
        \sum_i \abs{\tau(p_i^\varepsilon) -\tau(p_{i+1}^\varepsilon)} \leq \d(p_\varepsilon, q_\varepsilon) + \varepsilon.
    \]
    On the other hand, as the curve is piecewise causal, the left hand side coincides with the \(L_{\d}\)-length of the curve \(\tilde{\gamma}_\varepsilon\). By adding the two vertical segments joining \(p\) to \(p_\varepsilon\) and \(q_\varepsilon\) to \(q\) we get a curve \(\gamma_\varepsilon\) such that
    \[
        L_{\d}(\gamma_\varepsilon) = 2\varepsilon + L_{\d}(\tilde{\gamma}_\varepsilon) \leq 3\varepsilon  + \d(p_\varepsilon,q_\varepsilon).
    \]
    As \(\varepsilon \to 0\), the term \(\d(p_\varepsilon,q_\varepsilon)\) converges to \(\d(p,q)\). In particular the sequence of curves has uniformly bounded \(L_\d\) length. Up to reparametrization we can then assume that the curves are all parametrized over the interval \([0,1]\) with uniformly bounded metric speed. In particular the sequence \(\{\gamma_\varepsilon\}_\varepsilon\) is equi-Lipschitz with respect to \(\d\). As the metric completion is compact, we can apply Ascoli-Arzelà Theorem and get, up to subsequence, a uniform limit Lipschitz curve \(\gamma\) joining \(p\) to \(q\). We now notice that by lower semicontinuity of the length functional \(L_\d\) with respect to pointwise convergence we get that
    \[
        L_\d(\gamma) \leq \liminf_{\varepsilon \to 0} L_\d(\gamma_\varepsilon) \leq \d(p,q).
    \]
    As the opposite inequality is always true, we have shown that \(\gamma\) is a $\d$-minimizing Lipschitz curve in the metric completion joining \(p\) to \(q\). We now aim to show that \(\gamma\) is indeed a causal curve. To do so we notice that, for any \(t<s \in [0,1]\), we have
    \[
        \begin{split}
            t_q - t_p &= \d(p,q) = L_\d(\gamma) \\
            &=L_\d(\gamma_{\vert [0,t]}) +  L_\d(\gamma_{\vert [t,s]}) + L_\d(\gamma_{\vert [s,1]}) \\
            & \geq \d(p, \gamma(t)) + \d(\gamma(t), \gamma(s)) + \d(\gamma(s),q) \\
            & \geq \abs{\tau(\gamma(t))-\tau(p)} + \abs{\tau(\gamma(s)) - \tau(\gamma(t))} +  \abs{\tau(q) - \tau(\gamma(s))} \\
            &\geq \abs{t_q - t_p} = t_q - t_p.
        \end{split}
    \]
    This implies that all the previous estimates are actually equalities; in particular we have proven that
    \begin{equation}\label{eq:curve_is_basically_causal}
        L_\d(\gamma_{\vert [t,s]}) =\d(\gamma(t), \gamma(s)) = \abs{\tau(\gamma(t)) - \tau(\gamma(s))}.
    \end{equation}
    Setting \(t=0\) we see that
    \[
        L_\d(\gamma_{\vert [0,s]}) =\d(p, \gamma(s)) = \abs{\tau(\gamma(s))-t_p}.
    \]
    Notice that on the left hand side we have a non-decreasing, non-negative function; moreover when \(s=1\) the right hand side equals \(t_q -t_p\). This implies that \(\tau(\gamma(s))-t_p\) is never negative: if it were, by continuity it would vanish at some later time since for \(s=1\) it is positive, whereas by monotonicity the left hand side would be non-zero. Hence we can write
    \[
        \tau(\gamma(s)) = t_p + L_\d(\gamma_{\vert [0,s]}),
    \]
    which shows that \(\tau(\gamma(s))\) is a non-decreasing function along the curve. As \(\gamma\) is a Lipschitz curve, we can reparametrize it by arclength, so that it is never locally constant. With such a reparametrization we actually get that \(\tau\) is strictly increasing along the curve and in particular for \(t \in (0, \d(p,q))\) the curve is contained in the original spacetime. As the null distance encodes causality, from \eqref{eq:curve_is_basically_causal} we have that for each \(t < s \in (0,\d(p,q))\) it holds \(\gamma(t) \leq \gamma(s)\). Since the spacetime is globally hyperbolic, we can apply \cite[Lemma~2.21]{Kunzinger2018} and conclude that the curve is causal in the classical sense (almost everywhere). To replace this curve by a smooth one, we can take points along it converging to \(p\) and \(q\) respectively and join them by geodesics, which exist by global hyperbolicity. The same limit curve argument will provide a curve between \(p\) and \(q\) which will be a geodesic and thus smooth in the interior of the spacetime.
\end{proof}

We could wonder whether or not sequences as in \Cref{thm:future_developed_convergence} converge to a limit space which is causally-null. To answer this question, we firstly need to understand what the causally-null distance on the limit space is going to be. As the time function on the limit space is just the cosmological time function of the sequence, we only need to study the causal relation \(\leq_{d_\infty,\tau}\), which is done in the following proposition:

\begin{proposition}\label{prop:synthetic-causal-relation-vs-intersection}
    Let \((X,\d_j,\tau)\) be a sequence of timed metric spaces with increasing distances and assume that they converge pointwisely to a distance $d_\infty$. The causal relation \(\leq_{d_\infty, \tau}\) defined on the limit space \((X,d_\infty,\tau)\) coincides with the intersection of all causal relations \(\leq_{\d_j,\tau}\).
\end{proposition}

\begin{proof}
    For any two points \(p,q \in X\), we know that \(p \leq_{d_\infty,\tau} q\) if and only if
    \[
        d_\infty(p,q) = \tau(q) - \tau(p).
    \]
    Since \(d_\infty\) is the limit of the increasing sequence of distances \(\d_j\), for each \(j \in \N\) it holds that
    \[
        d_\infty(p,q) \geq \d_j(p,q) \geq \tau(q) - \tau(p),
    \]
    where the second inequality follows as \(\tau\) is \(1\)-Lipschitz for all distances \(\d_j\). From here it is immediate to see that \(p \leq_{d_\infty,\tau} q\) if and only if \(p \leq_{\d_j,\tau} q\) for all \(j \in \N\).
\end{proof}
This tells us that the causally-null distance \(\d_{d_\infty, \tau}\) on the limit space is the null distance computed with respect to the intersection of all causal relations, using the cosmological time function \(\tau\). 

On the other hand, one can define another causal relation on the limit space: with the assumption \eqref{eq:B}, the sequence of metrics \(g_j\) is converging pointwise to a non necessarily smooth tensor \(g_\infty\). Regardless of smoothness, this tensor induces a Lorentzian metric on each tangent space, as the vector \(\partial_t\) is of constant Lorentzian squared norm \(-1\) for all the metrics \(g_j\) and thus for \(g_\infty\) as well; as vectors that are spacelike for at least one of the metric\(g_j\) will be spacelike for the tensor \(g_\infty\), the limit metric will have signature \((-, + \dots, +)\) at each point. Following \cite[Def. 2.3]{grant2020} (building up on \cite{burtscher2015} for the Riemannian case) we define a causal curve in \([0,\tau_{\max}] \times M\) to be an absolutely continuous curve \(\gamma\) such that \(\dot{\gamma}\) is causal almost everywhere with respect to the metric \(g_\infty\) and future directed if \(g_\infty(\dot\gamma,\partial_t)<0\) almost everywhere. We can then define the causal relation \(\leq_\infty\) in the same fashion as smooth spacetimes:
\begin{definition}\label{def:absolutely-continuous}
    Let $g_\infty$ be the possibly non-continuous metric tensor obtained as the pointwise limit of the metric tensors $g_j$. We define the associated causal relation $\leq_\infty$ in $N$ as 
    \[
    p\leq_\infty q \iff \exists \gamma \text{ absolutely continuous future directed causal curve from }p \text{ to } q.
    \]
\end{definition}

In \cite{grant2020}, following \cite{chrusciel_lorentzian_2012}, it is shown that for metrics that are at least Lipschitz continuous, in particular for the metrics $g_j$, this notion of causality coincides with the classical one defined by means of piecewise smooth causal curves. For $g_\infty$ one can then define the associated null distance \(\d_\infty\) as in \Cref{def:null-distance} with respect to the cosmological time \(\tau\) by means of this causal relation. 

In the following proposition we show that \(\d_\infty\) is indeed a distance and the associated metric space to \((N,g_\infty)\) is causally-null. Moreover, we also show that the causal relation \(\leq_{\d_\infty,\tau}\) induced by \(\d_\infty\) coincides with the intersection of all the causal relations \(\leq_{\d_j,\tau}\) induced by the distances \(\d_j\), as happened with \(\leq_{d_\infty,\tau}\):
\begin{proposition}\label{prop:causal-relation-induced-by-limit-tensor}
    Let \((N,g_j)\) be a sequence of cosmological spacetimes satisfying assumption \eqref{eq:B} and denote by \(\d_\infty\) the null distance associated to the limit tensor \(g_\infty\). Then the completed space \((\bar{N}, \d_\infty)\) is causally-null, i.e., \(\d_\infty = \d_{\d_\infty,\tau}\).  Moreover, \(\leq_{\d_\infty,\tau} = \bigcap_j \leq_{\d_j,\tau}\), so that \(\leq_{\d_\infty,\tau} = \leq_{d_\infty, \tau}\).
\end{proposition}
\begin{proof}
    From the inequalities $g_1\leq g_\infty\leq \tilde{g}$ we deduce that $\leq_1\supseteq\leq_\infty$ and $\leq_\infty\supseteq \leq_{\tilde{g}}$, where we recall that all causal relations are built by means of absolutely continuous curves. Consequently, we have $\d_1\leq \d_\infty \leq \d_{\tilde{g}}$; in particular \(\d_\infty\) is a definite distance. Since we know that \(\d_1\) and \(\d_{\tilde{g}}\) are Lipschitz equivalent, we have that \(\d_\infty\) is Lipschitz equivalent to them as well. In particular it induces the manifold topology and, analogously to what is shown in \Cref{thm:compactness_nullcompletion}, the completions are bi-Lipschitz equivalent, and therefore they have the same topology. Lastly, we can extend the relation \(\leq_\infty\) to the relation \(\leq_{\d_\infty, \tau}\) on the metric completion. By the same reasoning as in \Cref{thm:cosmological_spacetimes_are_causally_null} one can see that the completed space \((\bar{N}, \d_\infty)\) is causally-null.

    For the second part of the statement, we firstly see that the inclusion \(\leq_{\d_\infty,\tau} \subset \bigcap_j \leq_{\d_j,\tau}\) is immediate: if \(p \leq_\infty q\) we have that \(\hat{d}_\infty(p,q) = \tau(q) - \tau(p)\). Since \(\d_j \leq \d_\infty\) we have that \(\hat{d}_j(p,q) = \tau(q) - \tau(p)\) and thus \(p \leq_{\d_j,\tau} q\) for all \(j \in \N\). Conversely, if \(p \leq_{\d_j,\tau} q\) for all \(j \in \N\) by \Cref{lemma:causal_relation_completion_implies_causal_curves} there exists a sequence of $C^0$ curves \(\{\gamma_j\}_j\) joining \(p\) to \(q\) such that each \(\gamma_j\) is \(g_j\)-causal in the interior of the spacetime. As the metrics are increasing, we have that the sequence \(\{\gamma_j\}_j\) is \(g_1\)-causal in the interior; moreover it is easy to see that endowing the manifold \(N\) with the Riemannian metric tensor \(s \coloneqq dt^2 +h_1(t,x)\) all of these curves, parametrized over the interval \([t_p,t_q]\) with the coordinate \(t\) as their parameter, are Lipschitz with constant \(\sqrt{2}\) and contained in the compact set  \([t_p, t_q] \times M\). This allows us to apply the limit curve theorem (cf. \cite[Theorem~3.1]{Minguzzi_2008}) and find an absolutely continuous future directed causal curve \(\gamma\) which is the uniform limit (with respect to the distance induced by \(s\)) of a subsequence of \(\{\gamma_j\}_j\). Since the original sequence had the property that for each \(k \in \N\) the subsequence \(\{\gamma_j\}_{j \geq k}\) is a sequence of \(g_k\) causal curves, for any \(k \in \N\) we could consider a tail of the sequence we extracted by applying the limit curve theorem with respect to the metric \(g_1\) such that all of the curves are \(g_k\)-causal; applying once again the limit curve theorem we would get a subsequence converging to a \(g_k\)-causal absolutely continuous curve. Since the first subsequence we extracted was already converging to \(\gamma\) we deduce that the limit is the same and that \(\gamma\) is \(g_k\)-causal for all \(k \in \N\), which means that \(g_k(\dot\gamma,\dot\gamma)\leq 0\) almost everywhere. Taking the limit as \(k \to +\infty\) we get that the curve is \(g_\infty\)-causal as well and thus \(p \leq_\infty q\), which then immediately implies \( p \leq_{\d_\infty, \tau} q\).
\end{proof}
As an immediate application we can prove \Cref{thm:causally_null_limit_is_null_distance_of_limit_tensor}:
\begin{proof}[Proof of {\Cref{thm:causally_null_limit_is_null_distance_of_limit_tensor}}]
    In the previous proposition we proved that \(\leq_{\d_\infty,\tau} = \leq_{d_\infty, \tau}\). In particular, we have \(\d_{\d_\infty,\tau} = \d_{d_\infty,\tau}\). On the other hand we know that \((\bar{N}, \d_\infty)\) is causally-null and thus \(\d_\infty = \d_{\d_\infty,\tau}\).
\end{proof}
\section{Examples}
    In this last section we discuss the results we obtained by applying them to some examples and showing where the assumptions are needed in order for them to hold. 
    
    More precisely, in \Cref{ex:allen-burtscher} we apply \Cref{thm:future_developed_convergence} to an example featured in \cite[Example 5.7]{allen_properties_2022} to show that one can conclude uniform and future developed convergence of a sequence of smooth spacetimes without necessarily requiring that their metric tensor converge uniformly. Next in \Cref{ex:limit_non_compact} we show that by dropping the assumptions made in \Cref{thm:future_developed_convergence} one can possibly end up with a limit space whose topology is different than the one of the sequence and, in particular, is non-compact. In \Cref{ex:triangle} we also show that without the assumption in \Cref{prop:causally_null_completions} we can build causally-null compactifiable spacetimes with completions that are not causally-null and in \Cref{ex:AB-bis} we show that even a uniform limit of causally-null spaces need not be causally-null. Lastly, in \Cref{ex:Cantor} we see that for non-smooth tensors as in \Cref{prop:causal-relation-induced-by-limit-tensor} it is necessary to lower the regularity class of causal curves to absolutely continuous ones for \Cref{prop:causal-relation-induced-by-limit-tensor} to hold. 
\begin{example}\label{ex:allen-burtscher}
    We apply our result to one example in the warped product case presented in \cite[Example 5.7]{allen_properties_2022}. This is an instance where one cannot apply \cite[Theorem 1.4]{allen_properties_2022} to study the convergence of the distances as the warping functions are not converging uniformly; however \Cref{thm:uniform_convergence} still applies:
    consider a compact Riemannian manifold \((M,h)\) and a smooth increasing function \(u \colon [0,1] \to (0,1]\) with \(u \equiv u_0 \in (0,1)\) on the interval \([0,1/2]\) and \(u(1) =1\), with \(u'(1) = 0\). We then consider the Lorentzian manifold \(N =(0,2) \times M\) with metrics
    \[
        g_j = -dt^2 +f_j(t)h,
    \]
    where
    \[
        f_j(t) \coloneqq
        \left\{\begin{aligned}
            &u(jt) \quad &&\text{if} \ t \in \left[0, \frac{1}{j}\right], \\
            &1 \quad &&\text{if} \ t \in \left(\frac{1}{j},1\right].
        \end{aligned}\right.
    \]
    Notice that the metrics \(h_j = f_j(t)h\) are a sequence of increasing metrics which are uniformly bounded by the metric \(h\). Then \Cref{thm:future_developed_convergence} applies and we get future developed convergence; moreover the corresponding null distances converge uniformly, even though the warping functions do not converge uniformly. Indeed, they are converging to the function
    \[
        f_\infty(t) \coloneqq
        \left\{\begin{aligned}
            &u_0\quad &&\text{if} \ t=0, \\
            &1 \quad &&\text{else},
        \end{aligned}\right.
    \]
    which is discontinuous. 
\end{example}

The condition $h_{j,t}\leq \tilde{h}_t$ featured in assumption \eqref{eq:B} is necessary: without it  we can build sequences that converge to a limit space with a different, non-compact topology. We build such an example in \Cref{ex:limit_non_compact} starting from a sequence of warped products. To prove that the limit is indeed not compact we firstly need some technical tools to explicitly compute the null distance in the warped product case. We start with the following lemma:

\begin{lemma}\label{lem:no-minimal-curve}
    Let $(M,h)$ be a compact connected Riemannian manifold and let $N=(0,\tau_{\max})\times M$ have Lorentzian metric $g=-dt^2+f(t)^2\,h$, where the warping function $f\colon (0,\tau_{\max})\to \R^+$ is smooth and strictly increasing. If $p,q\in N$ are such that $J^-(p)\cap J^-(q) = \varnothing$, then there is no piecewise causal curve joining $p$ and $q$ which has minimal null length.
\end{lemma}
\begin{proof}
    Consider two points $p_i=(t_i,x_i)\in N$ with $t_1 \leq t_2$ that are not causally related. Let us show that if $\gamma$ is a piecewise causal curve joining them with minimal null length, then the maximum value of its $t$-component is equal to $t_2$. 

    Indeed, assume otherwise that $\gamma=(\alpha,\beta)\colon [0,1]\to N$ is a piecewise causal curve joining $p_1$ and $p_2$ with null length $\L(\gamma)=\d(p_1,p_2)$ and $ t^*\coloneqq \max t(\gamma(s)) > t_2$ (see \Cref{fig:example}). Consider one of the points $p^*=\gamma(s^*)$ in the image of $\gamma$ such that $t(p^*)=t^*$. Let $p^-=\gamma(s^-)$ be the last point before $p^*$ with $t$-coordinate $t_2$, and $p^+=\gamma(s^+)$ the first point after $p^*$ with $t$-coordinate $t_2$. In particular, $t(\gamma(s))>t_2$ for every $s\in(s^-,s^+)$. Moreover, $\gamma|_{(s^-,s^*)}$ is future directed and $\gamma|_{(s^*,s^+)}$ is past directed. As $\gamma$ is piecewise causal (in fact piecewise null by \cite[Lemma~3.20]{sormani_null_2016}), one has
\[
    \begin{split}
        \L(\gamma|_{(s^-,s^+)})&=\int_{s^-}^{s^+} f(\alpha(s))\, \sqrt{h\bigl(\dot{\beta}(s),\dot{\beta}(s)\bigr)}\, ds \\
        &> f(t_2)\, d_M(x^-,x^+) = d_{t_2}(x^-,x^+) \geq \d(p^-,p^+),
    \end{split}
\]
where in the last inequality we used \Cref{thm:generalised_upper_bound}.
        
Now, the null length of $\gamma$ is equal to the sum of the lengths of its segments, i.e., 
\[
\hat{L}(\gamma)=\hat{L}(\gamma|_{(0,s^-)})+\hat{L}(\gamma|_{(s^*,s^+)})+\hat{L}(\gamma|_{(s^+,1)}) > \d(p,p^-)+\d(p^-,p^+)+\d(p^+,q) \geq \d(p,q),
\]
which contradicts the minimality of $\gamma$. With this argument we have also proved that if the breakpoint between two consecutive segments of a piecewise causal curve is in their future, then there exists a strictly shorter piecewise causal curve with the same endpoints. 

As a consequence, if a piecewise causal curve joining the two points is minimal, it cannot have more than one breakpoint: if it had at least two, there would be two consecutive segments that are respectively future and past directed and thus the previous argument would yield that the curve is not minimal. As the points are not causally related, the curve needs to have exactly one breakpoint. If the two segments are future and past directed, respectively, the previous argument applies again. Thus, the segments need to be past and future directed, respectively; however, this would contradict the assumption that the causal pasts do not intersect.
\end{proof}

We next give a formula to compute the null distance on warped products with an increasing warping function. This improves the one featured in \cite[Example~5.7]{allen_properties_2022} by providing an explicit expression in terms of the warping function and the given points.

\begin{proposition}\label{prop:explicit_null_distance}
    Let $(M,h)$ be a compact connected Riemannian manifold. Consider the manifold $N=(0,\tau_{\max})\times M$ with Lorentzian metric $g=-dt^2+f(t)^2\,h$, where the warping function $f\colon (0,\tau_{\max})\to \R^+$ is smooth and (non necessarily strictly) increasing. If $p,q\in N$ are such that $A\coloneqq J^-(p)\cap J^-(q) = \varnothing$, then
    \begin{equation}\label{eq:explicit_null_distance}
    \d(p,q)=t_p+t_q+\lim_{s\to 0} f(s)\, \biggl(d_M(x_p,x_q)- \int_0^{t_p}\frac{ds}{f(s)} - \int_0^{t_q}\frac{ds}{f(s)} \biggr),
    \end{equation}
    and in particular the integrals are finite. If, on the other hand, $A\neq\varnothing$, then $t$ achieves a maximum in $A$ at a point $r$, and
    \[
    \d(p,q)=t_p+t_q-2t_r.
    \]
\end{proposition}

\begin{proof}
First of all, the finiteness of the integrals is a consequence of \cite[Lemma~3.23]{sormani_null_2016}, which characterizes when points in a warped product are causally related: a point $r=(t_r,x_r)$ is in the intersection of the causal pasts of $p$ and $q$ if and only if
\[
d_M(x_p,x_r)\leq \int_{t_r}^{t_p}\frac{ds}{f(s)} \qquad \text{and} \qquad d_M(x_q,x_r)\leq \int_{t_r}^{t_q}\frac{ds}{f(s)}.
\]
In particular, if there is some $t^*\in(0,\tau_{\max})$ such that 
\[
\int_{t^*}^{t_p} \frac{ds}{f(s)} +
\int_{t^*}^{t_q} \frac{ds}{f(s)}  
\geq d_M(x_p,x_q),
\]
then there exists a point $x^*\in M$ such that $r\coloneqq(t^*,x^*)$ is in the intersection of the causal pasts of $p$ and $q$. Indeed, consider a minimal curve $\beta$ in $M$ joining $x_p$ and $x_q$. If any of the integrals above alone (say, without loss of generality, the one involving $p$) is greater than or equal to $d_M(x_p,x_q)$, take any point $x^*$ on $\beta$ such that the other integral (in our case, the one involving $q$) is greater than $d_M(x^*,x_q)$. For the other case, when both integrals are strictly smaller than $d_M(x_p,x_q)$, consider $x^*$ to be the unique point on $\beta$ such that $d_M(x_p,x^*)$ equals the integral above involving $p$.

So from $J^-(p)\cap J^-(q)=\varnothing$ one deduces that there is no $t^*$ as above. Thus
\begin{equation}\label{eq:integral_ineq}
\int_{0}^{t_p}\frac{ds}{f(s)} + \int_{0}^{t_q}\frac{ds}{f(s)}\leq d_M(x_p,x_q)<\infty,
\end{equation}
and in particular both integrals are finite. The converse implication, namely that \eqref{eq:integral_ineq} implies that the intersection of the causal pasts is empty, is easily seen to be true by using triangular inequality.

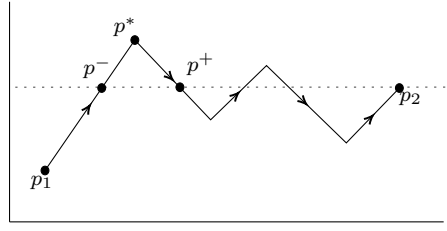
\begin{figure}
    \centering
    \begin{tikzpicture}[x=0.75pt,y=0.75pt,yscale=-1,xscale=1]
    \draw [color={rgb, 255:red, 74; green, 74; blue, 74 }  ,draw opacity=1 ] [dash pattern={on 0.84pt off 2.51pt}]  (441.4,45.8) -- (222.6,45.85) ;
    \draw    (223,3) -- (223,113.4) ;
    \draw    (440.6,113.4) -- (223,113.4) ;
    \draw  [fill={rgb, 255:red, 0; green, 0; blue, 0 }  ,fill opacity=1 ] (238.81,87.64) .. controls (238.81,86.44) and (239.68,85.47) .. (240.77,85.47) .. controls (241.85,85.47) and (242.73,86.44) .. (242.73,87.64) .. controls (242.73,88.83) and (241.85,89.8) .. (240.77,89.8) .. controls (239.68,89.8) and (238.81,88.83) .. (238.81,87.64) -- cycle ;
    \draw  [fill={rgb, 255:red, 0; green, 0; blue, 0 }  ,fill opacity=1 ] (416.44,46.27) .. controls (416.44,45.07) and (417.31,44.1) .. (418.4,44.1) .. controls (419.48,44.1) and (420.36,45.07) .. (420.36,46.27) .. controls (420.36,47.46) and (419.48,48.43) .. (418.4,48.43) .. controls (417.31,48.43) and (416.44,47.46) .. (416.44,46.27) -- cycle ;
    \draw    (240.77,87.64) -- (285.8,22.2) -- (323.8,62.2) -- (351.8,35) -- (391.8,73.8) -- (418.4,46.27) ;
    \draw [shift={(263.28,54.92)}, rotate = 124.54] [color={rgb, 255:red, 0; green, 0; blue, 0 }  ][line width=0.75]    (4.37,-1.96) .. controls (2.78,-0.92) and (1.32,-0.27) .. (0,0) .. controls (1.32,0.27) and (2.78,0.92) .. (4.37,1.96)   ;
    \draw [shift={(304.8,42.2)}, rotate = 226.47] [color={rgb, 255:red, 0; green, 0; blue, 0 }  ][line width=0.75]    (4.37,-1.96) .. controls (2.78,-0.92) and (1.32,-0.27) .. (0,0) .. controls (1.32,0.27) and (2.78,0.92) .. (4.37,1.96)   ;
    \draw [shift={(337.8,48.6)}, rotate = 135.83] [color={rgb, 255:red, 0; green, 0; blue, 0 }  ][line width=0.75]    (4.37,-1.96) .. controls (2.78,-0.92) and (1.32,-0.27) .. (0,0) .. controls (1.32,0.27) and (2.78,0.92) .. (4.37,1.96)   ;
    \draw [shift={(371.8,54.4)}, rotate = 224.13] [color={rgb, 255:red, 0; green, 0; blue, 0 }  ][line width=0.75]    (4.37,-1.96) .. controls (2.78,-0.92) and (1.32,-0.27) .. (0,0) .. controls (1.32,0.27) and (2.78,0.92) .. (4.37,1.96)   ;
    \draw [shift={(405.1,60.03)}, rotate = 134.01] [color={rgb, 255:red, 0; green, 0; blue, 0 }  ][line width=0.75]    (4.37,-1.96) .. controls (2.78,-0.92) and (1.32,-0.27) .. (0,0) .. controls (1.32,0.27) and (2.78,0.92) .. (4.37,1.96)   ;
    \draw  [fill={rgb, 255:red, 0; green, 0; blue, 0 }  ,fill opacity=1 ] (267.24,46.27) .. controls (267.24,45.07) and (268.11,44.1) .. (269.2,44.1) .. controls (270.28,44.1) and (271.16,45.07) .. (271.16,46.27) .. controls (271.16,47.46) and (270.28,48.43) .. (269.2,48.43) .. controls (268.11,48.43) and (267.24,47.46) .. (267.24,46.27) -- cycle ;
    \draw  [fill={rgb, 255:red, 0; green, 0; blue, 0 }  ,fill opacity=1 ] (306.44,45.87) .. controls (306.44,44.67) and (307.31,43.7) .. (308.4,43.7) .. controls (309.48,43.7) and (310.36,44.67) .. (310.36,45.87) .. controls (310.36,47.06) and (309.48,48.03) .. (308.4,48.03) .. controls (307.31,48.03) and (306.44,47.06) .. (306.44,45.87) -- cycle ;
    \draw  [fill={rgb, 255:red, 0; green, 0; blue, 0 }  ,fill opacity=1 ] (283.84,22.2) .. controls (283.84,21) and (284.72,20.04) .. (285.8,20.04) .. controls (286.88,20.04) and (287.76,21) .. (287.76,22.2) .. controls (287.76,23.4) and (286.88,24.36) .. (285.8,24.36) .. controls (284.72,24.36) and (283.84,23.4) .. (283.84,22.2) -- cycle ;
        
    \draw (232.4,89.2) node [anchor=north west][inner sep=0.75pt]  [font=\scriptsize]  {$p_{1}$};
    \draw (417.2,46.7) node [anchor=north west][inner sep=0.75pt]  [font=\scriptsize]  {$p_{2}$};
    \draw (274,9.5) node [anchor=north west][inner sep=0.75pt]  [font=\scriptsize]  {$p^{*}$};
    \draw (310.8,27.5) node [anchor=north west][inner sep=0.75pt]  [font=\scriptsize]  {$p^{+}$};
    \draw (258.8,29.1) node [anchor=north west][inner sep=0.75pt]  [font=\scriptsize]  {$p^{-}$};
    \end{tikzpicture}
    \caption{Situation described in the proof of \Cref{lem:no-minimal-curve}. The point $p^*$ has greater $t$-component than those of $p_1$ and $p_2$, from which one deduces that the piece of the curve from $p^-$ to $p^+$ cannot be minimal. Therefore, the whole curve cannot be minimal.}
    \label{fig:example}
\end{figure}

We proceed now to prove equation~\eqref{eq:explicit_null_distance} in the case $J^-(p)\cap J^-(q)=\varnothing$. To that end, let again $\gamma=(\alpha,\beta)\colon[0,1]\to N$ be a piecewise causal curve joining two points $p$ and $q$ whose pasts do not intersect. We will prove that there exists a piecewise null curve $\tilde{\gamma}$ joining the same points, which is not longer than $\gamma$, and whose first and last segments $\tilde{\gamma}_1$ and $\tilde{\gamma}_n$ have minimal $t$-component $t^*$ strictly smaller than the minimal $t$-component $t_{\min}$ of $\gamma$.

First, we show that $\gamma$ can be assumed to have null first and last segments. Indeed, for the first segment notice that $\gamma$, which is continuous, starts in the closed set $J^-(p)\cup J^+(p)$ and eventually leaves it. Therefore, there is a maximal value $s$ such that $p_1\coloneqq\gamma(s)\in J^-(p)\cup J^+(p)$. Moreover, $p_1$ is in the topological boundary of $J^-(p)\cup J^+(p)$ and, by standard causality theory, there is a null curve joining $p$ with $p_1$ which is therefore $\d$-minimizing. To sum up, the concatenation of such a curve with $\gamma|_{[s,1]}$ is a piecewise causal curve joining $p$ and $q$ whose first segment is null and it is not longer than $\gamma$. Similarly one proceeds with the last segment.

So, without loss of generality, assume that $\gamma$ is piecewise causal with first and last segments being null. Let $t^*\in(0,t_{\min})$ and define $\tilde{\gamma}=(\tilde{\alpha},\beta)$ in the following way: first we set its spatial part to be the same as that of $\gamma$. About the time component, we define it so that the first segment, $\tilde{\gamma}_1$, of $\tilde{\gamma}$ is the past directed null curve starting at $p$ whose endpoint has $t$-component $\smash{\tilde{\alpha}(\tilde{s}_1)=t^*}$ and the last segment, $\smash{\tilde{\gamma}_n}$, is the (future directed) null curve whose endpoint is $q$ and whose starting point has $t$-component $\smash{\tilde{\alpha}(\tilde{s}_{n-1})=t^*}$. These curves can be constructed by solving the corresponding ODE as in the proof of \Cref{thm:generalised_upper_bound}. For the remaining part of the curve $\tilde{\gamma}$, one can take any piecewise null curve joining $\tilde{\gamma}(\tilde{s}_1)$ with $\tilde{\gamma}(\tilde{s}_{n-1})$ with $t$-component between $t^*$ and $t_{\min}$.

The existence of these two segments (or, more precisely, that their parameters do not overlap, i.e., $\tilde{s}_1<\tilde{s}_{n-1}$) is a consequence of \eqref{eq:integral_ineq}:
    \[
    L_h(\beta)\geq d_M(x_p,x_q)>\int_{t^*}^{t_p}\frac{ds}{f(s)} + \int_{t^*}^{t_q}\frac{ds}{f(s)}= L_h(\beta|_{[0,\tilde{s}_1)}) +L_h(\beta|_{(\tilde{s}_{n-1},1]}).
    \]
 In other words, the length of $\beta$ is big enough for the null curve to go down from $t_p$ to $t^*$ (it ``spends" a spatial length equal to the first integral above) and then up from $t^*$ to $t_q$ (spatial length equal to the second integral).

By construction, $\tilde{\alpha}(s)<\alpha(s)$, whenever $s\in[\tilde{s}_1,\tilde{s}_{n-1}]$ whereas, being both $\tilde{\gamma}$ and $\gamma$ piecewise null, one has that $\tilde{\alpha}(s)\leq\alpha(s)$ in the first and last segments of $\tilde{\gamma}$.

Comparing the null lengths of $\gamma$ and $\tilde{\gamma}$ is now immediate:
    \[
    \hat{L}(\gamma) \geq
    \int_a^b f(\alpha(s)) \sqrt{h\bigl(\dot{\beta}(s),\dot{\beta}(s)\bigr)}\,ds\geq
    \int_a^b f(\tilde{\alpha}(s)) \sqrt{h\bigl(\dot{\beta}(s),\dot{\beta}(s)\bigr)}\,ds=\hat{L}(\tilde{\gamma}),
    \]
so $\tilde{\gamma}$ is indeed not longer than $\gamma$. We can estimate from below the length of the curve $\tilde{\gamma}$ by dividing it into three parts: its first segment, its last one, and the rest. We have
    \[
    \begin{split}
        \hat{L}(\tilde{\gamma})&=
        \hat{L}(\tilde{\gamma}_1) + \hat{L}(\tilde{\gamma}_{\text{rest}})+\hat{L}(\tilde{\gamma}_n)= 
        t_p-t^* + \int_{\tilde{s}_1}^{\tilde{s}_{n-1}} f(\tilde{\alpha}(s)) \sqrt{h\bigl(\dot{\beta}(s),\dot{\beta}(s)\bigr)} \, ds + t_q-t^* \\
        & \geq t_p+t_q-2t^* + f(t^*) \, L_h(\beta|_{[\tilde{s}_1,\tilde{s}_{n-1}]}).
    \end{split}
    \]
The length of the restriction of $\beta$ to $[\tilde{s}_1,\tilde{s}_{n-1}]$ is the total length of $\beta$ minus the parts corresponding to the first and last segments, i.e., 
    \[
        L_h(\beta|_{[\tilde{s}_1,\tilde{s}_{n-1}]})=
        L_h(\beta)-\int_{t^*}^{t_p}\frac{ds}{f(s)}-\int_{t^*}^{t_q}\frac{ds}{f(s)}.
    \]
Putting everything together and knowing that $d_M(x_p,x_q)\leq L_h(\beta)$, we have 
    \[
    \hat{L}(\gamma)\geq t_p+t_q-2t^* + 
     f(t^*) \biggl(d_M(x_p,x_q)- \int_{t^*}^{t_p}\frac{ds}{f(s)} - \int_{t^*}^{t_q}\frac{ds}{f(s)}
    \biggr).
    \]

As this estimate holds for every $t^*\in(0,t_{\min})$, we can let $t^*\to 0^+$. Taking the infimum over all piecewise causal curves joining $p$ and $q$ one deduces
\[
    \d(p,q) \geq t_p+t_q+\lim_{s\to 0} f(s)\, \biggl(d_M(x_p,x_q)- \int_0^{t_p}\frac{ds}{f(s)} - \int_0^{t_q}\frac{ds}{f(s)} \biggr).
\]

For the converse inequality, let $\beta$ be a geodesic in $M$ between $x_p$ and $x_q$. For $2t^*<\min\{t_p,t_q\}$, we build a piecewise causal curve $\tilde{\gamma}=(\alpha,\beta)$ as before: we take the first and last segments to be null, with minimal $t$-component $t^*$ and, respectively, past-directed and future-directed; the rest of the curve is piecewise causal with $t$-component between $t^*$ and $2t^*$. Then
\[
    d_M(x_p,x_q)- \int_{t^*}^{t_p}\frac{ds}{f(s)} - \int_{t^*}^{t_q}\frac{ds}{f(s)}=L_h(\beta_{|[\tilde{s}_1,\tilde{s}_{n-1}]})\geq d_M(\tilde{x}_1,\tilde{x}_{n-1}).
\]
Multiplying the right hand side by $f(t^*)$ we get by \Cref{thm:generalised_upper_bound}
\[
    f(t^*)\, d_M(\tilde{x}_1,\tilde{x}_{n-1})=d_{t^*}(\tilde{x}_1,\tilde{x}_{n-1})
    \geq \d(\tilde{\gamma}(\tilde{s}_1),\tilde{\gamma}(\tilde{s}_{n-1})),
\]
and by the triangle inequality,
    \[
    \begin{aligned}
    \d(p,q)&\leq \d\bigl(p,\tilde{\gamma}(\tilde{s}_1)\bigr)+ \d\bigl(\tilde{\gamma}(\tilde{s}_1),\tilde{\gamma}(\tilde{s}_{n-1})\bigr) + \d\bigl((\tilde{\gamma}(\tilde{s}_{n-1}),q\bigr) \\
    &\leq t_p+t_q-2t^* + f(t^*) \biggl(d_M(x_p,x_q)- \int_{t^*}^{t_p}\frac{ds}{f(s)} - \int_{t^*}^{t_q}\frac{ds}{f(s)}
    \biggr),
    \end{aligned}
    \]
which, by the arbitrariness of $t^*>0$, concludes the proof of the case $J^-(p)\cap J^-(q)=\varnothing$.

For the case $J^-(p)\cap J^-(q)\neq\varnothing$ we have two ``subcases": either $p$ and $q$ are causally related or they are not. In the former case, say $p\leq q$, we have $t_r=t_p$ where we recall that $t_r$ is the maximum $t$-component in $A\coloneqq J^-(p)\cap J^-(q)=J^-(p)$. Thus, $t_p+t_q-2t_r=t_q-t_p$, which we know \cite[Lemma~3.11]{sormani_null_2016} coincides with the null distance between $p$ and $q$. 

So assume $A\neq \varnothing$ and that $p$ and $q$ are not causally related. Set $t_{\max}\coloneqq\max\{t_p,t_q\}$. Being $M$ compact, the time function $t$ is a regular cosmological time function by \Cref{lemma:regular_cosmotime_when_compact_slices}. In particular, $N$ is globally hyperbolic, which in turn implies that causal pasts and futures are closed. So consider a point $z=(t_z,x_z)\in A$ and the subset of $A$ of points whose $t$-component is greater than or equal to $t_z$. This set is closed and contained in the compact set $[t_z,t_{\max}]\times M$, so it is itself compact. The continuity of $t$ gives now the existence of a maximum in such a set, and thus in $A$. Call $t_r$ the maximum and $r\in A$ a point where it is attained.

It is clear now that $\d(p,q)\leq t_p+t_q-2t_r$, as the right hand side is the null length of the piecewise causal curve going down from $p$ to $r$ and then up from $r$ to $q$. For the converse inequality we would like to follow a similar reasoning as in the precedent part of the proof. The difference is that now
    \[
    d_M(x_p,x_q)=\int_{t_r}^{t_p}\frac{ds}{f(s)}+ \int_{t_r}^{t_q}\frac{ds}{f(s)},
    \]
so that a piecewise causal curve cannot be ``improved" by going as close to $t=0$ as wanted, as opposed to what happened in the case $A=\varnothing$.

So consider a piecewise causal null curve $\gamma=(\alpha,\beta)$ from $p$ to $q$. If its minimum $t$-value is lower than or equal to $t_r$, then $\hat{L}(\gamma)\geq t_p+t_q-2t_r$. On the other hand, if its minimum $t$-value is strictly greater than $t_r$ we build another piecewise null curve $\tilde{\gamma}=(\tilde{\alpha},\beta)$ in the same spirit as before, i.e., whose first segment goes from $p$ down until $\tilde{\alpha}(\tilde{s}_1)=t_r$ and whose last segment goes up to $q$ from $\tilde{\alpha}(\tilde{s}_{n-1})=t_r$. The rest of the curve (if there is anything left), can be chosen to be piecewise null with $t$-component between $t_r$ and $\min\alpha$. By the same comparison as before, one deduces that $t_p+t_q-2t_r\leq\hat{L}(\tilde{\gamma})\leq\hat{L}(\gamma)$. So also $\d(p,q)\geq t_p+t_q-2t_r$.

\end{proof}

\begin{remark}
    If in the setting of \Cref{prop:explicit_null_distance} the warping function is decreasing instead of increasing, the proof can be immediately adapted to show that if the set $B \coloneqq J^+(p)\cap J^+(q)$ is empty then 
    \[
        \d(p,q) =2\tau_{\max}-t_p-t_q+\lim_{s\to \tau_{\max}}f(s)\, \biggl(d_M(x_p,x_q)- \int_{t_p}^{\tau_{\max}}\frac{ds}{f(s)} - \int_{t_q}^{\tau_{\max}}\frac{ds}{f(s)} \biggr),
    \]
    and if \(B \neq \varnothing\), defining $t_r$ to be the minimum value of $t$ in $B$ we have
    \[
         \d(p,q) = 2t_r-t_p-t_q.
    \]
\end{remark}

\begin{example}\label{ex:limit_non_compact}
    Let $(M,h)$ be a compact Riemannian manifold and consider in $N=(0,\tau_{\max})\times M$ the sequence of Lorentzian metrics given by $g_j=-dt^2+h_j(t)$, where $h_j(t)=(1+tj)h$. The sequence of induced null distances $\d_j$ in $N$ is increasing by \cite[Lemma~4.15]{allen_properties_2022}. \Cref{thm:compactness_nullcompletion} ensures that the metric completion $(\bar{N}_j,\d_j)$ of each metric space $(N,\d_j)$ is compact, but we cannot apply \Cref{thm:future_developed_convergence} as the sequence of tensors is not uniformly bounded as in assumption \eqref{eq:B}. Indeed, the sequence of distances $\d_j$ in $N$ converges pointwisely to the distance $d_\infty$ given for $p=(t_p,x_p)$ and $q=(t_q,x_q)$, by
    \[
    d_{\infty}\bigr((t_p,x_p),(t_q,x_q)\bigl)=
    \begin{cases}
    t_p+t_q+d_M(x_p,x_q),  &x_p\neq x_q, \\
    \abs{t_p-t_q},  &x_p=x_q.
    \end{cases}
    \]
    In particular $(N_\infty,d_\infty)$, where $N_\infty=[0,\tau_{\max}]\times M$, is not compact. 
    
    \begin{proof}
    For $x_p=x_q$ we have $p\leq q$ or $q\leq p$ and thus $\hat{d}_j\bigr(p,q\bigl)=\abs{t_p-t_q}$, for every $j\in \N$. 
    
    So let us assume $x_p\neq x_q$. We will show that for $j$ big enough, their $j$-causal pasts do not intersect. Indeed, the intersection of their causal pasts is
    \[
    J_j^-(p)\cap J_j^-(q)=
    \biggl\{
    (t_r,x_r): d_M(x_p,x_r)\leq \int_{t_r}^{t_p}\frac{ds}{1+js}
    \quad \text{and} \quad
    d_M(x_q,x_r)\leq \int_{t_r}^{t_q}\frac{ds}{1+js}
    \biggr\}.
    \]
    The integrals are bounded above by $\ln(1+jt_p)/j$ or $\ln(1+jt_p)/j$, respectively, so it suffices to have $j$ big enough so that the sum of these bounds is strictly smaller than $d_M(x_p,x_q)$ to have an empty intersection.

    So for big enough $j$, we can apply the first part of \Cref{prop:explicit_null_distance} to have
    \[
    \d_j(p,q)=t_p+t_q+d_M(x_p,x_q)- \int_0^{t_p}\frac{ds}{1+js} - \int_0^{t_q} \frac{ds}{1+js} \xrightarrow{j\to\infty} t_p+t_q+d_M(x_p,x_q).
    \]

    This implies that $(N_\infty,d_\infty)$ is not compact. Indeed, consider a sequence $\{p_n=(t_n,x_n)\}\subset N$ such that $t_n=k>0$ is constant and the points $x_n\in M$ are all distinct. Then $d_\infty(p_n,p_m)>2k$, $\forall m\neq n$. In particular, such a sequence admits no $d_\infty$-converging subsequence, as no subsequence is even $d_\infty$-Cauchy.
    \end{proof}
\end{example}

\begin{example}\label{ex:triangle}
    In general, the extension of a null distance in a causally-null compactifiable spacetime need not be causally-null if the ``causal accessibility'' condition featured in \Cref{prop:causally_null_completions} is not assumed. Indeed, let $\varepsilon>0$ and consider in the Minkowski plane, with respect to \((t,x)\)-coordinates, the points $(0,-1)$, $(1,0)$ and $(0, 1+\varepsilon)$ (see \Cref{subfig:triangle}). Consider the (open) subset $\triangle$ of Minkowski plane bounded by the triangle formed by joining these points by straight lines. This is a causally-null compactifiable spacetime in which the null distance coincides with the restriction of the null distance in Minkowski plane, i.e., $\d(p,q)=\max\{\abs{\tau(p)-\tau(q)},\abs{x_p-x_q}\}$. In particular, this distance is bi-Lipschitz equivalent to the Euclidean distance and, therefore, the metric completion of the open triangle with respect to the null distance is just the closed triangle $\bar{\triangle}$. 

    The induced causal relation in the completion is given by 
    \[
    p\leq_{\d,\tau}q\iff \d(p,q)=\tau(q)-\tau(p).
    \]
    In particular, it coincides in $\triangle$ with the usual causal relation, something that we already knew because $\d$ encodes causality in $\triangle$. However, consider $r\coloneqq(1+\varepsilon,0)\in\bar{\triangle}\backslash\triangle$. It is clear from $\tau(r)=0$ and the expression of $\leq_{\d,\tau}$, that $p\leq_{\d,\tau}r \iff p=r$. For the converse relation we have
    \[
    r\leq_{\d,\tau}p\iff \d(p,r)=\max\bigl\{\tau(p),\abs{x_p-(1+\varepsilon)}\bigr\}=\tau(p)\iff \tau(p)+x_p\geq 1+\varepsilon.
    \]
    However, apart from $r$, no point in $\bar{\triangle}$ satisfies such condition, as the angle of the triangle at the point \(r\) is less than \(45\) degrees, whereas joining any point satisfying the condition above to the point \(r\) would make a (clockwise) angle with the \(x\)-axis bigger than \(45\) degrees and thus it would not be contained in \(\triangle\).

    As a consequence, if $p\neq r$ there is no $\leq_{\d,\tau}$-alternating sequence as in \Cref{def:causally_null_distance} joining $p$ and $r$. Therefore, the induced causally-null distance $\d_{\d,\tau}$ from $r$ is always infinite, whereas the extension of the original null distance to the metric completion was bounded. In other words, $\d\neq \d_{\d,\tau}$ in $\bar{\triangle}$. Notice that the point $r$ does not satisfy the accessibility condition required in \Cref{prop:causally_null_completions} for the completion of a causally-null compactifiable spacetime to be causally-null, namely, that each neighborhood $U$ of $r$ (in the topology of the metric completion) admits a point $p\in \triangle\cap U$ such that $p\leq_{\d,\tau}r$ or $r\leq_{\d,\tau}p$.
\end{example}

\begin{figure}[ht]
    \centering
        \tikzset{every picture/.style={line width=0.75pt}} 
        \begin{tikzpicture}[x=0.75pt,y=0.75pt,yscale=-1,xscale=1]
        \draw [color={rgb, 255:red, 155; green, 155; blue, 155 }  ,draw opacity=1 ]   (305.75,116.25) -- (305.83,226.17) ;
        \draw [color={rgb, 255:red, 155; green, 155; blue, 155 }  ,draw opacity=1 ]   (207.5,208.25) -- (443,207.75) ;
        \draw  [fill={rgb, 255:red, 0; green, 0; blue, 0 }  ,fill opacity=1 ] (304.13,136.17) .. controls (304.13,135.34) and (304.8,134.67) .. (305.63,134.67) .. controls (306.45,134.67) and (307.13,135.34) .. (307.13,136.17) .. controls (307.13,137) and (306.45,137.67) .. (305.63,137.67) .. controls (304.8,137.67) and (304.13,137) .. (304.13,136.17) -- cycle ;
        \draw    (305.63,136.17) -- (233.5,208.29) ;
        \draw    (422.63,207.67) -- (233.5,208.29) ;
        \draw    (305.63,136.17) -- (422.63,207.67) ;
        \draw  [fill={rgb, 255:red, 0; green, 0; blue, 0 }  ,fill opacity=1 ] (421.13,207.67) .. controls (421.13,206.84) and (421.8,206.17) .. (422.63,206.17) .. controls (423.45,206.17) and (424.13,206.84) .. (424.13,207.67) .. controls (424.13,208.5) and (423.45,209.17) .. (422.63,209.17) .. controls (421.8,209.17) and (421.13,208.5) .. (421.13,207.67) -- cycle ;
        \draw  [fill={rgb, 255:red, 0; green, 0; blue, 0 }  ,fill opacity=1 ] (232,208.29) .. controls (232,207.46) and (232.67,206.79) .. (233.5,206.79) .. controls (234.33,206.79) and (235,207.46) .. (235,208.29) .. controls (235,209.12) and (234.33,209.79) .. (233.5,209.79) .. controls (232.67,209.79) and (232,209.12) .. (232,208.29) -- cycle ;
        
        \draw (272.75,121.65) node [anchor=north west][inner sep=0.75pt]  [font=\footnotesize]  {$( 1,0)$};
        \draw (310.75,110.65) node [anchor=north west][inner sep=0.75pt]  [font=\footnotesize]  {$t$};
        \draw (197.75,190.65) node [anchor=north west][inner sep=0.75pt]  [font=\footnotesize]  {$( 0, -1)$};
        \draw (379.75,215.65) node [anchor=north west][inner sep=0.75pt]  [font=\footnotesize]  {$r=( 0, 1+\varepsilon)$};

        \draw (430.75,200.65) node [anchor=south][inner sep=0.75pt]  [font=\footnotesize]  {$x$};
        \end{tikzpicture}
        \caption{Situation described in \Cref{ex:triangle}. The point $r$, in the metric boundary of the triangle, is not related to any point in the completion with the causal relation $\leq_{\d,\tau}$. Therefore, $r$ is at infinite $\d_{\d,\tau}$-distance from any other point in $\bar{\triangle}$.}
        \label{subfig:triangle}
\end{figure}
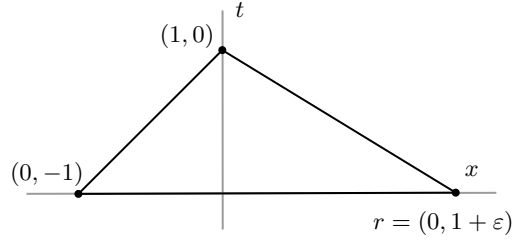

\begin{example}\label{ex:limit-relation-not-intersection}
    In \Cref{prop:causal-relation-induced-by-limit-tensor} one proves the equality $\bigcap_j\leq_j = \leq_\infty$ for the case of a sequence of cosmological spacetimes. In this example we show that the inclusion $\subseteq$ is not true for general spacetimes. Specifically, consider the subset $N$ of $\R^2$ obtained after erasing a horizontal closed half-line (see \Cref{subfig:causal-relations-comparison}; the same result can be achieved by removing just one point). Consider, for every $j\in\N$, the Lorentzian metric on $N$ given by $g_j\coloneqq-dt^2+(1- \frac{1}{2j}) dx^2$, and call $g_\infty\coloneqq-dt^2+dx^2$, i.e., $g_\infty$ is just the restriction of the Minkowski metric to $N$. The causal cones in this sequence are decreasing with $j$ (with respect to inclusion) and the ``limit cones'' are the usual ones. It is now immediate to check that the points $p,q$ in \Cref{subfig:causal-relations-comparison} are causally related with respect to every metric $g_j$, but they are unrelated with respect to $g_\infty$. In other words, $\bigcap_j\leq_j \not\subset \leq_\infty$. The issue here is that, even if for each $j\in\N$ there is a $g_j$-causal curve between \(p\) and \(q\), the sequence of them is not contained in any compact and the Euclidean distance (which is the one induced by the Riemannian metric tensor $s$ as as in the proof of \Cref{prop:causal-relation-induced-by-limit-tensor}) is not proper, so one cannot invoke the limit curve theorem as in the aforementioned proof. Notice, however, that the equality $\bigcap_j\leq_j = \leq_{d_\infty,\tau}$ from \Cref{prop:synthetic-causal-relation-vs-intersection} is still true. Furthermore, in this case the null distance $\d_\infty$ induced by the limit tensor $g_\infty$ coincides with the pointwise limit $d_\infty$ of the null distances $\d_j$.
\end{example}

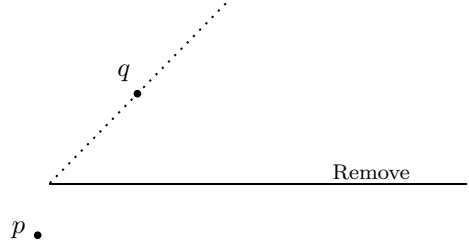
\begin{figure}
        \centering
        \tikzset{every picture/.style={line width=0.75pt}} 
        \begin{tikzpicture}[x=0.75pt,y=0.75pt,yscale=-1,xscale=1]
        
        \draw    (260.6,190.6) -- (470.2,190.6) ;
        \draw  [dash pattern={on 0.84pt off 2.51pt}]  (349.25,100) -- (260.6,190.6) ;
        
        \draw (401.2,179.4) node [anchor=north west][inner sep=0.75pt]  [font=\scriptsize] [align=left] {{Remove}};
        \draw (240.5,208.15) node [anchor=north west][inner sep=0.75pt]  [font=\small]  {$p$};
        \draw (293.5,129.15) node [anchor=north west][inner sep=0.75pt]  [font=\small]  {$q$};
        \draw  [fill={rgb, 255:red, 0; green, 0; blue, 0 }  ,fill opacity=1 ] (303.52,145.3) .. controls (303.52,144.52) and (304.15,143.89) .. (304.93,143.89) .. controls (305.7,143.89) and (306.33,144.52) .. (306.33,145.3) .. controls (306.33,146.08) and (305.7,146.71) .. (304.93,146.71) .. controls (304.15,146.71) and (303.52,146.08) .. (303.52,145.3) -- cycle ;
        \draw  [fill={rgb, 255:red, 0; green, 0; blue, 0 }  ,fill opacity=1 ] (253.52,216.3) .. controls (253.52,215.52) and (254.15,214.89) .. (254.93,214.89) .. controls (255.7,214.89) and (256.33,215.52) .. (256.33,216.3) .. controls (256.33,217.08) and (255.7,217.71) .. (254.93,217.71) .. controls (254.15,217.71) and (253.52,217.08) .. (253.52,216.3) -- cycle ;
        \end{tikzpicture}
    \caption{Situation described in \Cref{ex:limit-relation-not-intersection}. Points $p$ and $q$ are timelike related with respect to any Lorentzian metric $g_a=-dt^2+a^2dx^2$ with $a<1$, but they are not even causally related with respect to $g_1$.}
    \label{subfig:causal-relations-comparison}
\end{figure}

\begin{example}\label{ex:AB-bis}
    In this example we show that the uniform monotone limit of causally-null distances need not be causally-null. Indeed, consider again \cite[Example~5.7]{allen_properties_2022} as in \Cref{ex:allen-burtscher}. By \Cref{prop:causal-relation-induced-by-limit-tensor} we know that $\leq_\infty = \leq_{d_\infty,\tau}$ and thus $\d_{g_\infty}=\d_{d_\infty,\tau}$. However, in the aforementioned paper, the authors prove that $\d_{g_\infty}\neq d_\infty$. As a consequence, $d_\infty$ does not coincide with the distance it induces via \Cref{def:causally_null_distance} or, in other words, it is not causally-null.
\end{example}

\begin{example}\label{ex:Cantor}
    In \Cref{def:absolutely-continuous} we define $\leq_\infty$ by means of absolutely continuous curves which are almost everywhere $g_\infty$-causal and future directed and mention that, for Lipschitz metrics, this causal relation coincides with the classical one. In this example we show that the equality \(\leq_{\d_\infty,\tau} = \bigcap_j \leq_{\d_j,\tau}\) featured in \Cref{prop:causal-relation-induced-by-limit-tensor} would not be true if $\leq_\infty$ were defined by means of (piecewise) smooth curves. To that end, we prove that such an equality does not even hold in the interior, i.e., \(\leq_\infty \neq \bigcap_j \leq_j\). Consider the ``fat'' Cantor set $C_a$ obtained by removing from the segment $[0,1]$ the middle interval of length $a<1/3$ and recursively the middle intervals of length $a^n$ from the remaining intervals. This set is, as the usual Cantor set, closed with empty interior; however, it has measure $1-\sum_{n=1}^\infty a^n 2^{n-1}=\frac{1-3a}{1-2a} >0$. Now consider the characteristic function $\chi$ on $\R\backslash C_a$ and any positive number \(b>0\). As $C_a$ is closed, $1+b\chi$ is lower semicontinuous. One can approximate $1+b\chi$ from below by a monotone increasing sequence of smooth functions $f_j \geq 1$. Now, consider $(M,h)$ compact and Riemannian, as usual, and the sequence of smooth Lorentzian metric tensors $g_j=-dt^2+f_j(t)h$ on $N=(0,1)\times M$, i.e., each term is a warped product with warping function $f_j$. The limit tensor $g_\infty=-dt^2+(1+b\chi(t))h$ is discontinuous whenever $t\in C_a$. Consider now two numbers $t_1<t_2\in C_a$ such that they are not both the endpoints of the same erased segment from $[0,1]$. Then between $t_1$ and $t_2$ there are uncountably many points in $C_a$ and in fact \(C_a \cap (t_1,t_2)\) has positive measure. Finally, consider two points, $p=(t_1,x)$ and $q=(t_2,y)$, such that there exists an absolutely continuous curve $\gamma\colon [t_1,t_2]\to N$ from $p$ to $q$ satisfying $g_\infty(\dot\gamma,\dot\gamma)=0$ and $g_\infty(\partial_t,\dot\gamma)>0$, almost everywhere. In other words, $\gamma$ is a future directed $g_\infty$-null curve in the sense of \cite{grant2020}. Then $\gamma$ cannot be piecewise smooth, as in fact any representative of its derivative, extended to the full interval \([t_1,t_2]\), is discontinuous on a positive measure set of times (see \Cref{fig:cantor_cone}). Consequently, $p\leq_\infty q$ when the relation is defined by absolutely continuous curves, but not when defined by piecewise smooth ones. On the other hand, $p\leq_j q$, for every $j$.

    \begin{figure}[ht]
        \centering
        \includegraphics[scale = 0.5]{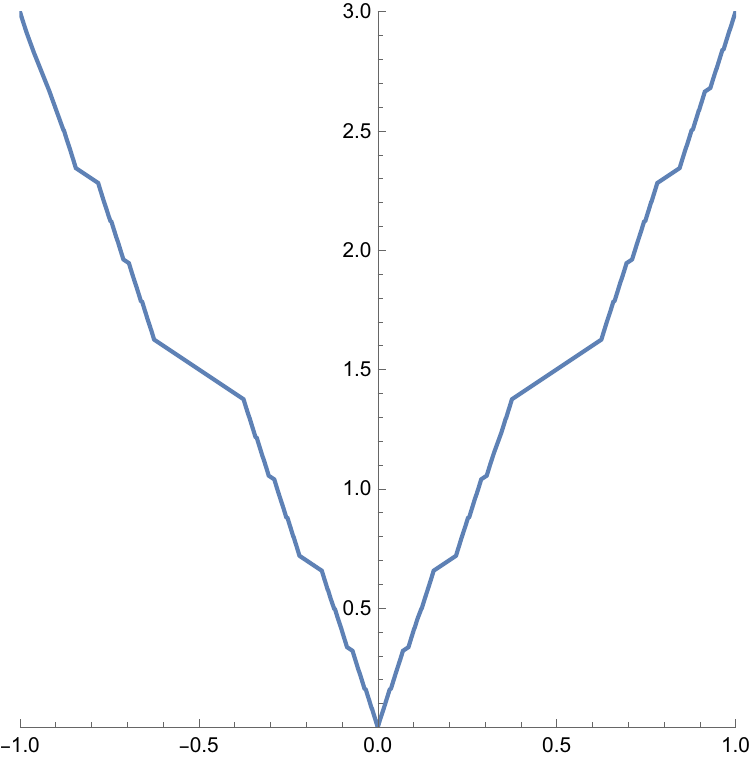}
        \caption{Future cone of the origin in the non-compact case \(M=\R\) with respect to the tensor \(g_\infty\) with parameters \(a =1/4\), \(b=4\). The boundary curves are null curves whose derivative is either \(1\) or \(5\). The set of times where the derivative ``jumps" has positive measure. To use this example for a compact spatial slice one can just consider $M=\mathbb{S}^1$.}
        \label{fig:cantor_cone}
    \end{figure}
\end{example}

\printbibliography[title=References]

\end{document}